\documentclass[11pt, reqno]{amsart}

\usepackage{amsmath,tikz}
\usepackage{amssymb}
\usepackage{fullpage}
\usepackage{amsthm}
\usepackage{graphicx}
\usepackage{placeins}
\usepackage{caption}
\usepackage{extpfeil}
\usepackage{enumerate}
\usepackage{caption}
\usepackage{subcaption}
\usepackage{hyperref}
\usepackage{xcolor}

 \numberwithin{equation}{section}
\newcommand{\be}{\begin{equation}}
\newcommand{\ee}{\end{equation}}
\newcommand{\bee}{\begin{equation*}}
\newcommand{\eee}{\end{equation*}}

\newcommand{\supp}{\textup{supp}\,}

\numberwithin{equation}{section}
\newtheorem{theorem}{Theorem}[section]
\newtheorem{remark}[theorem]{Remark}
\newtheorem{lemma}[theorem]{Lemma}
\newtheorem{definition}[theorem]{Definition}
\newtheorem{proposition}[theorem]{Proposition}
\newtheorem{corollary}[theorem]{Corollary}

\newcommand{\R}{\mathbb{R}}
\newcommand{\ds}{\displaystyle}

\begin{document}

\title{Uniqueness and non-uniqueness of steady states of aggregation-diffusion equations}

\date{}
\author{Matias G. Delgadino$^1$, Xukai Yan$^2$, and Yao Yao$^3$}
\thanks{$^1$Department of Mathematics, Pontifical Catholic University of Rio de Janeiro, Rio de Janeiro, RJ 22451-900 Brazil. Email: matias.delgadino@puc-rio.br. MGD was partially supported by EPSRC grant number EP/P031587/1.}
\thanks{$^2$School of Mathematics, Georgia Institute of Technology, 686 Cherry Street, Atlanta, GA 30332-0160 USA. E-mail: xukai.yan@math.gatech.edu}
\thanks{$^3$School of Mathematics, Georgia Institute of Technology, 686 Cherry Street, Atlanta, GA 30332-0160 USA. E-mail: yaoyao@math.gatech.edu. YY was partially supported by the NSF grant DMS-1715418 and DMS-1846745.}

\begin{abstract}
We consider a nonlocal aggregation equation with degenerate diffusion, which describes the mean-field limit of interacting particles driven by nonlocal interactions and localized repulsion. When the interaction potential is attractive, it is previously known that all steady states must be radially decreasing up to a translation, but uniqueness (for a given mass) within the radial class was open, except for some special interaction potentials. For general attractive potentials, we show that the uniqueness/non-uniqueness criteria are determined by the power of the degenerate diffusion, with the critical power being $m=2$. In the case $m\geq 2$, we show that for any attractive potential the steady state is unique for a fixed mass. In the case $1<m<2$, we construct examples of smooth attractive potentials, such that there are infinitely many radially decreasing steady states of the same mass. For the uniqueness proof, we develop a novel interpolation curve between two radially decreasing densities, and the key step is to show that the interaction energy is convex along this curve for any attractive interaction potential, which is of independent interest. 

\end{abstract}

\maketitle

\section{Introduction}
 In this paper, we study the uniqueness question of steady states of the aggregation equation with degenerate diffusion
\begin{equation}\label{eq:evolution}
    \partial_t \rho=\Delta\rho^m+\nabla\cdot(\rho\nabla( W*\rho)), \quad x \in \mathbb{R}^n, t\geq 0.
\end{equation}
Throughout this paper we assume $m>1$, and $W \in C^\infty(\mathbb{R}^n \setminus \{0\})$ is a radially symmetric \emph{attractive} potential, that is, $W'(r)>0$ for all $r>0$, where $r$ is the radial variable.

Equation \eqref{eq:evolution} arises as the mean-field limit of an interacting particle system driven by localized repulsion and pairwise nonlocal interaction \cite{O90,BodnarVelazquez2}, and it has been extensively studied in various contexts in math biology and physics. It appears in mathematical biology as a macroscopic model for collective animal behavior such as swarming \cite{MogilnerEdelstein, boi00, MogilnerEdelsteinBent, Morale2005, TopazBertozzi2, BURGER2007939}, and the assumption $m>1$ models the anti-overcrowding effect \cite{boi00, TopazBertozzi2, CC06}. In particular, when $W = \mathcal{N}$ is the attractive Newtonian potential in $\mathbb{R}^n$, \eqref{eq:evolution} is known as the Patlak-Keller-Segel model for chemotaxis, which describes the collective motion of cells attracted by a self-emitted chemical substance; see \cite{patlak, ks, Horstmann} and the references therein. The case $m=1$, $W=\mathcal{N}$ for dimensions $n=2,3$ is also called the Smoluchowski--Poisson system in gravitational physics \cite{CHAVANIS2004145, chavanis2010self}. Besides the above applications, \eqref{eq:evolution} also has applications in granular media \cite{BCP, cmcv-06} and material science \cite{PhysRevLett.95.226106}.

 Identifying the steady states of \eqref{eq:evolution} is a key step towards understanding the dynamics. The first step is existence of steady states, which has been settled by Lions' concentration-compactness principle and its variations. The next natural question is whether steady states are unique for a given mass. Recently, it was shown that for radial attractive potential $W$, all bounded steady states must be radially decreasing up to a translation. Thus to answer the uniqueness question, one only needs to focus on uniqueness within the class of radially decreasing densities. To the best of our knowledge, uniqueness results have been only established for the cases when either the interaction potential has some special homogeneity or convexity properties, or for the special diffusion power $m=2$.  Therefore, for a general attractive potential,  even though all steady states are known to be radially decreasing, it was an open question whether they are unique within this class for a given mass. The main point of this paper is to answer the uniqueness question in the positive for $m\ge2$ and in the negative for $1<m<2$.

\subsection{Previous results in the literature.}\label{subsec_review} In this subsection, we briefly summarize the previous results in the literature regarding the gradient flow structure of \eqref{eq:evolution}, as well as the existence, symmetry and uniqueness results on the steady states.

\noindent $\bullet$ \emph{Gradient flow structure.} 
 Equation \eqref{eq:evolution} has an associated \emph{free energy} functional  $\mathcal{E}[\rho]$, which plays an important role in the study of steady states and well-posedness of solutions. It is defined as
\begin{equation}\label{def_e}
\mathcal{E}[\rho] = \frac{1}{m-1}\int_{\mathbb{R}^n} \rho^m dx + \frac{1}{2} \int_{\mathbb{R}^n} \rho (W*\rho) dx =: \mathcal{S}[\rho] + \mathcal{I}[\rho],
\end{equation}
where the \emph{entropy} $\mathcal{S}[\rho]$ corresponds to the nonlinear diffusion term (where $\mathcal{S}[\rho]$ becomes $\int_{\mathbb{R}^n} \rho \log \rho dx$ when $m=1$), and the \emph{interaction energy} $\mathcal{I}[\rho]$ corresponds to the nonlocal interaction term in \eqref{eq:evolution}.
%

 In fact, $\mathcal{E}$ is more than just a Lyapunov functional of \eqref{eq:evolution}. It plays a crucial role in this equation, since \eqref{eq:evolution} has a formal gradient flow structure. Denote by $\mathcal{P}_2(\mathbb{R}^n)$ the set of probability measures with finite second moment. If $\rho_0 \in \mathcal{P}_2(\mathbb{R}^n)$ with $\mathcal{E}[\rho_0]<\infty$, then formally speaking, $\rho(t)$ is the gradient flow of $\mathcal{E}$ in $\mathcal{P}_2(\mathbb{R}^n)$ endowed with the 2-Wasserstein distance $d_2$, in the sense that
\[
\partial_t \rho(t) = -\nabla_{d_2} \mathcal{E}[\rho(t)],
\]
for a generalized notion of gradient $\nabla_{d_2}$ induced by the 2-Wasserstein metric. This observation was first made by Otto in the seminal paper \cite{Otto} for the porous medium equation, and later generalized to a large family of equations, which can include a drift potential \cite{JKO}, or a nonlocal interaction term \cite{cmcv-03, cmcv-06}. 
For a comprehensive presentation of the theory of gradient flows in Wasserstein metric spaces, see the books \cite{AGS, Villani}.

\noindent $\bullet$  \emph{Existence of steady states.}
 To show the existence of steady states to \eqref{eq:evolution}, a natural idea is to look for the global minimizer of $\mathcal{E}$. The standard method of proof is the direct method of calculus of variations. That is to say, taking a minimizing sequence, showing that up to a subsequence it converges and proving lower semicontinuity of the functional along this sequence. Of course, this strategy fails if the energy is not bounded below. If $W$ is singular near the origin in the form $W \sim \frac{|x|^k}{k}$ with $k\in (-n,0)$, the Hardy-Littlewood-Sobolev inequality gives that the energy is bounded below in the \emph{diffusion-dominated regime} $m>1-k/n$.
 
 The harder step is to show the compactness of the minimizing sequence, namely that the mass of the minimizing sequence does not escape to infinity. Let us formally consider the scaling of the energy functional under the dilation $\rho_\lambda(x) = \lambda^n \rho(\lambda x)$ for $0<\lambda\ll 1$.  For potentials with non-integrable decay $W \sim \frac{|x|^k}{k}$ with $k\in (-n,0)$ for $|x|\gg 1$, we notice the scalings 
$$
\mathcal{S}[\rho_\lambda] = \lambda^{(m-1)n} \mathcal{S}[\rho]\qquad\mbox{ and}\qquad \mathcal{I}[\rho_\lambda] \sim \lambda^{-k} \mathcal{I}[\rho].
$$
Heuristically, we can see that spreading is not beneficial in terms of the energy if $(m- 1)n>-k$. In this case, it can be shown that a global minimizer exists. The first rigorous proof was done by Lions' concentration-compactness principle \cite{lions, lions2} for the Newtonian potential in the diffusion dominated regime $m>2-2/n$, and the same proof also works for growing interaction potentials $\lim_{|x|\to\infty}W(x)=\infty$ for all $m>1$. This result was generalized in \cite{bedrossian1} to obtain existence of a global minimizer for interaction potentials with decay $W \sim \frac{|x|^k}{k}$ for $k\in(-n,0)$ in the regime $m>1-k/n$.

\noindent For integrable potentials with $\int_{\mathbb{R}^n} W dx < \infty$, the energy scales as follows for the dilation $\rho_\lambda$ with $\lambda\ll 1$ (see \cite[Section 2.3.1]{CCY} for a derivation):
\begin{equation}\label{eq:scalingintegrable}
\mathcal{E}[\rho_\lambda] =\lambda^{(m-1)n} \int_{\mathbb{R}^n} \frac{1}{m-1}\rho^m dx  + \frac{\lambda^n}{2} \left(\int_{\mathbb{R}^n} W dx\right) \int_{\mathbb{R}^n} \rho^2 dx + o(\lambda^n).    
\end{equation}
Comparing the two powers of $\lambda$ suggests that $m = 2$ is the critical power separating the cases where it is favorable ($m < 2$) versus unfavorable ($m > 2$) for the mass to spread to infinity. This is indeed reflected in the following rigorous results. For $m>2$, for any attractive potential $W$, there exists a global minimizer for any given mass \cite{bedrossian1, BFH14}.  At $m=2$, existence versus non-existence of the global minimizer depends on the value of $\int_{\mathbb{R}^n} W dx$. Namely, for any mass, a global minimizer exists if and only if $\int_{\mathbb{R}^n} W dx \in [-\infty, -2)$ \cite{bedrossian1, BDF}.
 For $1<m<2$, there is a global minimizer if $-\int_{\mathbb{R}^n} W dx$ is sufficiently large \cite{kaib}.

\noindent For potentials with growth $\lim_{|x|\to\infty} W(x)=+\infty$, there is a global minimizer for all $m>1$ for any given mass, see \cite{CCV15, CDP}. For $m=1$, \cite{CDP} identified a sharp condition on $W$ that distinguishes existence/non-existence of global minimizers, $W$ has to grow at least logarithmically at infinity to provide confinement of the mass.

\noindent In all the above cases, the global minimizer of (1.2) corresponds to a steady state to (1.1) in the sense of distributions.  In addition, the global minimizer must be radially decreasing due to Riesz's rearrangement inequality.

\noindent $\bullet$ \emph{Radial symmetry of steady states.} In the cases that a steady state is known to exist, it is natural to ask whether it must be radially symmetric (note that a steady state is not necessarily the global minimizer). For all radial attractive potential $W$ that are no more singular than the Newtonian potential near the origin, \cite{CHVY} showed that every bounded steady state must be radially decreasing up to a translation for all $m>0$. This result was later generalized to more singular Riesz potentials $\mathcal{W}_k := \frac{|x|^k}{k}$, $k\in (-n,2-n)$ in \cite{CHMV}. These radial symmetry results show that when the repulsion is modeled by \emph{local} (linear/nonlinear) diffusion, symmetry breaking cannot happen in steady states. As a contrast, when repulsion is modeled by \emph{nonlocal} interaction via an attractive-repulsive interaction potential, numerical and analytical evidence  \cite{Kolokolnikov,  BKSUV,  BalagueCarrilloLaurentRaoul_Dimensionality}  shows that steady states can develop non-radial patterns, despite the radial symmetry of the interaction potential.

\noindent $\bullet$  \emph{Uniqueness/non-uniqueness of steady states.} For a given mass (without loss of generality we set the mass be 1 in the rest of this paper, and let $\mathcal{P}(\mathbb{R}^n)$ denote the probability densities), uniqueness of steady states is only known in special cases. In fact, even the uniqueness of the global minimizer is only known in the following cases:

If $W(x)$ is convex, it is known that the gradient flow is a contraction in the 2-Wasserstein sense \cite{AGS, Villani}, leading to uniqueness of steady states. 

For the attractive Newtonian potential $\mathcal{N}$, uniqueness of steady states among radial functions was first shown by \cite{LY} in the diffusion-dominated regime $m>2-2/n$. This result was generalized to potentials of the form $W=\mathcal{N}*h$ for a radially decreasing function $h$ by \cite{KY}. For Riesz  potentials $\mathcal{W}_k:=\frac{|x|^k}{k}$ with $k\in(-n,n)$ (where $W_0$ is replaced by $\ln|x|$) in the diffusion-dominated regime $m>1-k/n$, uniqueness of steady states is obtained for dimension $n=1$ in \cite{CHMV}, and for $n\geq 1$ in the recent work \cite{CCH19}, where both proofs strongly rely on the homogeneity properties of Riesz potentials.

 For general attractive potentials, uniqueness was only known for the special case when the diffusion power is $m=2$, under the additional assumption $W\in C^1(\mathbb{R}^n)$. Note that when $m=2$, both terms in \eqref{eq:evolution} are quadratic in $\rho_s$. As a result, if $\rho_s \in \mathcal{P}(\mathbb{R}^n)\cap L^\infty(\mathbb{R}^n)$ is a steady state supported in $B(0,R)$, then it satisfies
\[
2\rho_s+W*\rho_s=C \quad \textrm{ in } B(0,R),
\] 
where the left hand side is a linear operator of $\rho_s$. Using the Krein-Rutman theorem, (where the linearity of the left hand side in $\rho_s$ is crucial), \cite{BDF} proved uniqueness of steady states for $m=2$, under the additional assumption $W\in C^1(\mathbb{R})$. This result was generalized by \cite{kaib} to $\mathbb{R}^n$.

To the best of our knowledge, among all radial attractive potentials, so far the only non-uniqueness example is when $W=C_{k,n}\mathcal{W}_k$ is a multiple of Riesz potential with $k\in(-n,n)$, and the equation is in the \emph{fair-competition case}  $m=1-k/n \in(0,2)$. In this case, it is known that there exists a unique $C_{k,n}>0$ such that \eqref{eq:evolution} has a one-parameter family of steady states in $\mathcal{P}(\mathbb{R}^n)$, where all of them are dilations of each other, and they can be characterized as optimizers of a variant of Hardy-Littlewood-Sobolev inequalities \cite{BCL,BRB,CCH17}.

\subsection{Our results}

Throughout this paper, we assume that $m> 1$, and $W$ satisfies the following assumptions. Note that these assumptions cover all Riesz potentials $\mathcal{W}_k=\frac{|x|^k}{k}$ with $k\in (-n, 1]$.

\smallskip

\noindent (W1) $W(x) \in C^\infty(\mathbb{R}^n\setminus\{0\})$ is radially symmetric, and $W$ is an attractive potential, i.e. $W'(r)>0$ for all $r>0$, where $r$ is the radial variable.

\noindent (W2) $W$ is no more singular than some locally integrable Riesz potential $\frac{|x|^{k}}{k}$ at the origin for some $k>-n$, in the following sense:  $W'(r) \leq C_w r^{k-1}$ for all $r\in(0,1)$ for some $k>-n$ and $C_w>0$.

\noindent (W3) There exists some $C_w > 0$ such that $W'(r) \leq C_w$ for all $r>1$.

\noindent (W4) Either $W(r)$ is bounded for $r\geq 1$, or there exists some $C_w>0$ such that for all $a,b\geq 0$ we have
\[
W_+(a+b) \leq C_w (1+W(1+a) + W(1+b)),
\]
where $W_+ := \max\{W, 0\}$.

Here the assumptions (W3) and (W4) control the growth of $W$ for $r>1$, and they are only to ensure that we can apply the previous results to show that all steady states are radially decreasing and compactly supported (see Lemma~\ref{lemma_stat} for the precise statement). Other than in Lemma~\ref{lemma_stat}, these two assumptions play no role in the uniqueness proof. We note that (W4) is usually referred as the doubling hypothesis.

Our first main result is the uniqueness result of steady states for  a general attractive potential for $m\geq 2$. For the precise definition of steady states, see Definition~\ref{def:gendef}.

\begin{theorem}\label{thm:uniqueness}
Assume $m\geq 2$, and $W$ satisfies \textup{(W1)}--\textup{(W4)}. Then \eqref{eq:evolution} has at most one steady state in $\mathcal{P}(\mathbb{R}^n) \cap L^\infty(\mathbb{R}^n)$ up to a translation.\end{theorem}

\noindent\textbf{Remark.} The result and the proof of Theorem~\ref{thm:uniqueness} can be directly generalized to nonlinear diffusion operators of the form $\nabla\cdot (\rho \nabla \Phi'(\rho))$, where $\Phi:\R^+\to \R^+$ is a strictly increasing smooth function with $\Phi'''\geq0$.

Our second main result is the non-uniqueness result for $1<m<2$. Since it is known that steady states are unique for all $m\geq 1$ for certain potentials (such as convex potentials, or Riesz potentials in the diffusion-dominated regime), it is impossible to obtain a non-uniqueness result for all attractive potentials. That being said, we can still show that the non-uniqueness result for $1<m<2$ is rather generic, in the following sense. Given any attractive potential $W_0$ in the diffusion dominated regime (for technical reasons, we also assume it is no more singular than the Newtonian potential at the origin), and an arbitrarily large $R_0>0$, we can always modify the the tail of $W_0$ in $B(0,R_0)^c$ into a new attractive potential $\widetilde W$, such that \eqref{eq:evolution} has infinitely many radially decreasing steady states in $\mathcal{P}(\mathbb{R}^n) \cap L^\infty(\mathbb{R}^n)$ after the modification.

\begin{theorem}\label{thm:nonuniqueness}
Assume $m\in (1,2)$, $W_0$ satisfies \textup{(W1)} and \textup{(W2)} with $k>-n(m-1)$ and $k\geq -n+2$. Then for any $R_0>0$, there exists an attractive potential $\widetilde W \in C^\infty(\mathbb{R}^n\setminus \{0\}) \cap W^{1,\infty}(B(0,R_0)^c)$ that is identical to $W_0$ in $B(0,R_0)$, such that \eqref{eq:evolution} with potential $\widetilde W$ has infinitely many radially decreasing steady states in $\mathcal{P}(\mathbb{R}^n) \cap L^\infty(\mathbb{R}^n)$.
\end{theorem}

In particular, if we let $W_0$ be a smooth attractive potential (such as $|x|^2$) in the above theorem, it immediately leads to the following corollary.

\begin{corollary}
For any $m\in (1,2)$, there exists a smooth attractive potential $\widetilde W \in C^\infty(\mathbb{R}^n) \cap W^{1,\infty}(\mathbb{R}^n)$, such that \eqref{eq:evolution} has infinitely many radially decreasing steady states in $\mathcal{P}(\mathbb{R}^n) \cap L^\infty(\mathbb{R}^n)$.
\end{corollary}

 We point out that our non-uniqueness result are only in the degenerate diffusion regime $m\in(1,2)$. Its proof strongly relies on the fact that the steady states are compactly supported for $m>1$, thus does not apply to the $m\in (0, 1]$ case. This leads to the following open question:

\noindent\textbf{Open problem.}  In the case of $m\in (0,1]$,  are radially symmetric steady states of \eqref{eq:evolution} unique in $\mathcal{P}(\mathbb{R}^n)$ for a general attractive potential $W$?

%

\subsection{Strategy of proof} We first describe our method of proof for the uniqueness result in Theorem~\ref{thm:uniqueness}.
 For equations with a gradient flow structure, steady states need to be critical points of the energy functional. If for any two critical points, one can construct a smooth curve connecting them such that the energy along this curve is strictly convex, then there cannot be more than one critical points. 
This natural but powerful idea was formulated into a general Banach framework in \cite{bonheure2018paths} and here we adapt it to the case of the Wasserstein metric. 

 Of course, when trying to apply this interpolation argument to \eqref{eq:evolution}, the main question is how to find an interpolation curve along which the energy is convex, if it exists at all. Since  \eqref{eq:evolution} is formally a gradient flow in $\mathcal{P}_2$ with 2-Wasserstein metric, a natural candidate is the geodesic of this space. 
However, in general $\mathcal{E}$ is not convex along a geodesic for non-convex $W$. Note that other common interpolations (such as linear interpolation) fail for \eqref{eq:evolution} with general attractive potentials as well. Convexity along linear interpolation coincides with the Fourier transform satisfying $\hat W(\xi)\geq 0$ for all $\xi \ne0$, see \cite{lopes}.

 The key idea of our approach is a novel interpolation curve between any two radially decreasing functions. Let us start with a heuristic explanation; the rigorous definition are postponed to Section~\ref{sec_2}. If $\rho_0$, $\rho_1$ are both step functions having $N$ horizontal layers with mass $1/N$ in each layer, then we define the interpolation curve $\rho_t$ by deforming each layer so that its height changes linearly, and meanwhile adjust the width so that the mass in each layer remains constant. Figure~\ref{fig_steps} illustrates the interpolation for step functions with two layers. We can similarly define such interpolation between any two radially decreasing functions in $\mathcal{P}(\mathbb{R}^n)$, which can be seen as a $N\to\infty$ limit of the step function case.

\begin{figure}[h!]
\begin{center}
\includegraphics[scale=1]{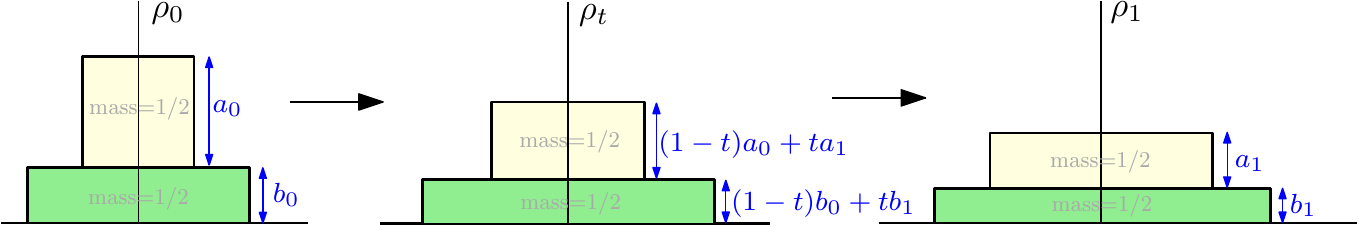}
\end{center}
\caption{A brief sketch of the interpolation curve $\{\rho_t\}_{t\in[0,1]}$ when $\rho_0$, $\rho_1$ are both step functions with two layers, where the mass of each layer is $1/2$.\label{fig_steps}}
\end{figure}

 Note that this interpolation curve $\rho_t$ is not the linear interpolation between $\rho_0$ and $\rho_1$, and it is not the geodesic in 2-Wasserstein metric either. We are unaware of any previous results using such an interpolation. While the interpolation may a-priori seem unnatural, it enjoys a remarkable property that the interaction energy $\mathcal{I}[\rho_t]$ is strictly convex along this curve for \emph{all attractive potentials} satisfying (W1) and (W2). (No convexity assumption on the potential is needed.) Regarding the entropy $\mathcal{S}[\rho_t]$, a simple argument gives that it is convex along this curve if and only if $m\geq 2$. In addition, we will show the curve is Lipschitz in 2-Wasserstein distance. Thus the convexity of $\mathcal{E}[\rho_t]$ for $m\geq 2$ leads to the uniqueness of steady states. Here the proof is shorter if \eqref{eq:evolution} has a rigorous gradient flow structure, which we describe in Section~\ref{sec_flow}. Without a rigorous gradient flow structure, we can still obtain uniqueness via a longer approach in Section~\ref{sec_general}, where we need to establish some fine regularity properties of the curve in Section~\ref{sec:further_reg}.

In the case $m\in(1,2)$, the entropy $\mathcal{S}[\rho_t]$ fails to be convex under the interpolation curve described above, thus our uniqueness proof does not apply. In fact, as shown by Theorem~\ref{thm:nonuniqueness}, $m=2$ is indeed the threshold separating uniqueness/non-uniqueness of steady states. To obtain non-uniqueness for $m\in(1,2)$, we take a potential $W_1$ with a steady state $\rho_s^1$ supported in some ball $B(0,R_1)$ and we modify its tail so that the new potential is still attractive and it also admits a new steady state.

In particular, the new attractive potential $W_2$ we construct is identical to $W_1$ in $B(0,2R_1)$, so that $\rho_s^1$ remains a steady state. To obtain another steady state, we set $W_2'(r)\equiv \epsilon$ for $r>3R_1$, where $0<\epsilon\ll 1$ is a sufficiently small constant. Note that \emph{if} we had set $\epsilon=0$, we would have $W_2\equiv \text{const}$ for $|x|>3R_1$, thus $W_2$ becomes an integrable potential after subtracting a constant. For integrable potentials, the scaling limit \eqref{eq:scalingintegrable} suggests that the long term dynamics of \eqref{eq:evolution} should be significantly different for the regimes $m\in(1,2)$ and $m\in(2,\infty)$. In the case $m\in(1,2)$, if the initial data is sufficiently flat it is energetically favorable for the solution to keep flattening. While in the case $m\in[2,\infty)$, it is not. We quantitatively study this phenomenon by tracking the evolution of the $L^{3-m}$ norm. However, instead of setting $\epsilon=0$, we set $0<\epsilon\ll 1$ so that a flat enough initial data remains flat for all time, but cannot spread to infinity due to the linear growth of $W_2$ as $|x|\to\infty$. For such $W_2$ with flat enough initial data, we show that along a diverging subsequence of time, the dynamical solution converges to a new steady state $\rho_s^2$ that is necessarily flatter than $\rho_s^1$. Finally, once we have that $W_2$ has two radially decreasing steady states, we can use an inductive argument to construct a potential with infinitely many steady states.

\subsection{Notations} \label{sec_notation}Throughout this paper, we use $\|f\|_p$ to denote the $L^p$ norm of $f$.
We denote by $c_n$ the volume of unit ball in $\mathbb{R}^n$, and $\omega_n$ the surface area of $(n-1)$-dimensional unit sphere in $\mathbb{R}^n$. Let $\mathcal{P}(\mathbb{R}^n)$ denote the probability densities in $\mathbb{R}^n$, and $\mathcal{P}_2(\mathbb{R}^n)$ denotes the probability densities  in $\mathbb{R}^n$ with finite second moment.

For a set $D \subset \mathbb{R}^n$, let $1_D(x)$ be the indicator function of $D$. We will often consider the special case where $D = B(0,r)$ is the ball centered at 0 with radius $r$, and to simplify the notation, we let $\chi_r(x) := 1_{B(0,r)}(x)$.

\section{Definition and properties of the interpolation curve}\label{sec_2}

\subsection{Definition of the height function}\label{sec_def_h}

As we explained in the introduction, the proof of the uniqueness result requires a novel interpolation using the ``height function with respect to mass''. Let $\rho\in \mathcal{P}(\R^n) \cap L^\infty(\R^n)$ be radially decreasing. (In Lemma~\ref{lemma_stat} we will show that all steady states belong to this class up to a translation.) For such $\rho$, we define its associated \emph{height function} $h(s):(0,1)\to (0,\|\rho\|_\infty)$ implicitly by
\begin{equation}\label{def_h_eq}
\int_{\R^n}\min\{\rho(x),h(s)\}\,dx=s.
\end{equation}
The definition of $h$ is illustrated in Figure~\ref{fig_h}(a): As we use a horizontal plane to cut the region below the graph of $\rho$, for any given $s\in (0,1)$, $h(s)$ is the (unique) height of the plane such that the mass below the plane is equal to $s$. See Figure~\ref{fig_h}(b) for a sketch of $h$. In the following lemma we prove some properties of $h$.

\begin{figure}[h!]
\begin{tikzpicture}
    \begin{scope}[xshift=-7cm]
    \node {\includegraphics[scale=0.9]{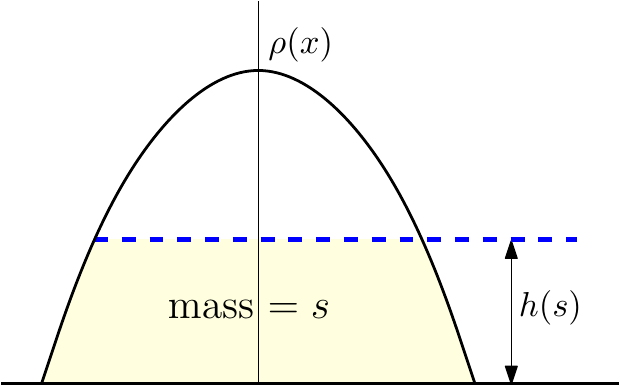}};
    \end{scope}
      \begin{scope}[xshift=-2.2cm, yshift=-0.13cm]
    \node {\includegraphics[scale=0.9]{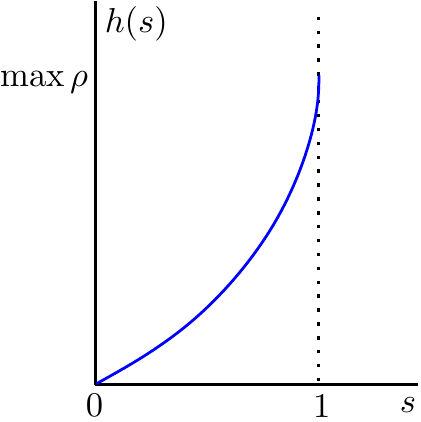}};
    \end{scope}
    \begin{scope}[xshift=3cm]
    \node {\includegraphics[scale=0.9]{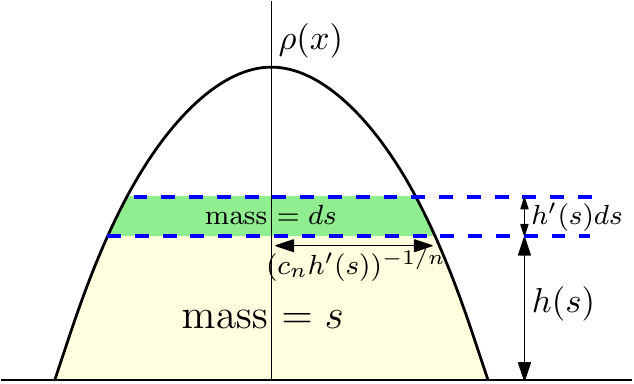}};
    \end{scope}
\end{tikzpicture}

\hspace*{-1cm}(a) \hspace{5cm} (b) \hspace{3.5cm} (c)
\caption{(a) An illustration of the definition of the height function $h$. (b) A sketch of the function $h(s)$. (c) A graphical explanation of the relation \eqref{h_rho}: note that the green region is an infinitesimally short cylinder with volume $ds$ and height $h'(s)ds$, thus its radius must be $(c_n h'(s))^{-1/n}$. \label{fig_h}}
\end{figure}

\begin{lemma}\label{lemma_h}
For a radially strictly decreasing probability density $\rho\in  \mathcal{P}(\mathbb{R}^n) \cap L^\infty(\R^n)$, let its height function $h$ be given by \eqref{def_h_eq}. Then $h$ satisfies the following properties.

\begin{enumerate}[(a)]
\item $h(s) \in (0,\|\rho\|_\infty)$ for $s\in (0,1)$, is continuous, strictly increasing, and convex.  In addition, we have
\begin{equation}\label{eq_h'}
h'(s)=|\{\rho>h(s)\}|^{-1} \text{ for a.e. }s\in(0,1).
\end{equation}

\item $\rho$ is compactly supported if and only if $\lim_{s\to 0^+}h'(s)>0$. In addition, we have $\lim_{s\to 0^+}h'(s)=|\supp \rho|^{-1}$. 
\item The function $h$ fully determines $\rho$, in the sense that
\begin{equation}\label{h_rho}
\rho(x)=\int_0^1\chi_{(c_n h'(s))^{-1/n}}(x) h'(s)\,ds \quad\text{ for a.e. }x\in\mathbb{R}^n,
\end{equation}
where $\chi_r(x) := 1_{B(0,r)}(x)$ is as defined in Section~\ref{sec_notation}.
(See Figure~\ref{fig_h}(c) for a graphical explanation of \eqref{h_rho}.)
\item If in addition we assume that $\rho$ is strictly decreasing in the radial variable within its support, then $h \in C^1((0,1))$, and
$$
\lim_{s\to 1^-}h'(s)=+\infty.
$$

\end{enumerate}

\end{lemma}

\begin{proof}

 Given a fixed positive height $h>0$, the mass of $\rho$ under height $h$ is given by 
\begin{equation}\label{def_s_eq}
s(h):= \int_{\mathbb{R}^n}\min\{\rho(x),h\}\,dx.
\end{equation}
It follows that the function $s:(0,\infty)\to (0,1)$ is strictly increasing  in $h$ for $h \in (0,\|\rho\|_\infty)$, $s(0)=0$ and $\lim_{h\nearrow \|\rho\|_\infty}s(h)=1$. Note that for any $h>0, \delta>0$ with $h+\delta < \|\rho\|_\infty$, we have
\[
0 < |\{x\in\mathbb{R}^n: \rho>h+\delta\}| \leq \frac{s(h+\delta)-s(h)}{\delta} \leq |\{x\in\mathbb{R}^n: \rho(x)>h\}|
\]
Sending $\delta\to 0^+$ and using the upper bound, we have that $s(h)$ is continuous for $h \in (0,\|\rho\|_\infty)$. In addition, sending $\delta\to 0^+$ and using the fact that $\lim_{\delta\to0^+} |\{\rho>h+\delta\}|=|\{\rho>h\}|$, we have that the right derivative of $s$ satisfies
\begin{equation}\label{right_der}
\frac{d^+}{dh}s(h)=\lim_{\delta\to 0^+} \frac{s(h+\delta)-s(h)}{\delta} = |\{\rho>h\}| > 0 \qquad\mbox{for all $h \in (0, \|\rho\|_\infty)$}.
\end{equation}
Similarly, for the left derivative,
\begin{equation}\label{left_der}
 \frac{d^-}{dh}s(h)=\lim_{\delta\to 0^-} \frac{s(h+\delta)-s(h)}{\delta} = |\{\rho\geq h\}| > 0 \qquad\mbox{for all $h \in (0, \|\rho\|_\infty)$}.
\end{equation}
Hence, by the monotonicity of $|\{\rho> h\}|$, we obtain that $s$ is concave.

Comparing \eqref{def_s_eq} with \eqref{def_h_eq}, the function $h(s)$ in \eqref{def_h_eq} is the inverse of $s(h)$. Thus  $h$ is continuous, strictly increasing in $(0,1)$, convex, and satisfies $\lim_{s\to 0^+} h(s)=0$, $ \lim_{s\to 1^-} h(s)=\|\rho\|_{\infty}$.
In addition, \eqref{right_der} directly implies that the right derivative of $h$ satisfies 
\begin{equation}\label{eq_h_der}
\frac{d^+}{ds}h(s)=|\{\rho>h(s)\}|^{-1} \text{ for all }s\in(0,1).
\end{equation}
By monotonicity of $h$, we also know $h$ is differentiable a.e. in $(0,1)$, with $h'$ satisfying \eqref{eq_h'}.
 
To prove (c), we start with the identity 
\[
\rho(x) = \int_0^{\|\rho\|_\infty} 1_{\{\rho(x)>h\}} dh = \int_0^1 1_{\{\rho(x)>h(s)\}} h'(s) ds.
\]
Since $\rho$ is assumed to be radially decreasing, each level set $\{\rho>h\}$ is a ball centered at origin with radius $(c_n^{-1} |\{\rho>h\}|)^{-1/n}$, thus we have $\rho(x)>h$ if and only if $|x| < (c_n^{-1} |\{\rho>h\}|)^{1/n}$. Combining this with  \eqref{eq_h_der} gives 
\[
\rho(x) = \int_0^1 \chi_{ (c_n^{-1} |\{\rho>h(s)\}|)^{1/n}}(x) h'(s) ds = \int_0^1 \chi_{ (c_n h'(s))^{-1/n}}(x) h'(s) ds,
\]
finishing the proof of (c).

To show (d), note that if $\rho$ is strictly radially decreasing within its support, it implies $\{\rho>h\}=\{\rho\geq h\}$ for all $h>0$. Combining this fact with \eqref{right_der} and \eqref{left_der} gives that $s(h)\in C^1((0,\|h\|_\infty))$ with a strictly positive derivative, thus $h(s) \in C^1((0,1))$. In addition, $\rho_s$ being strictly radially decreasing near the origin implies that $\lim_{h\to\|\rho\|_\infty^-} |\{\rho>h\}| = 0$. Combining this with \eqref{eq_h_der} and the fact that $\lim_{s\to 1^-} h(s) = \|h\|_\infty$, we have $\lim_{s\to 1^-}h'(s)=\lim_{a\to \|h\|_\infty^-} |\{\rho>a\}|^{-1} = +\infty.$
%
%
%
%
%
\end{proof}

\subsection{Interpolation using the height function}
Let $\rho_0$ and $\rho_1$ be two radially symmetric and decreasing probability densities, with $h_0$ and $h_1$ as their associated height functions. We consider the curve $\{\rho_t\}_{t\in[0,1]}$ of radially symmetric decreasing probability densities, whose height function $h_t$ is a linear interpolation between $\rho_0$ and $\rho_1$:
\begin{equation}\label{h_t}
h_t(s):=(1-t)h_0(s)+th_1(s) \quad\text{ for } s\in(0,1), t\in[0,1],
\end{equation}
and $\rho_t$ is determined by its height function $h_t$ via the relation \eqref{h_rho}, that is,
\begin{equation}\label{rho_t}
\rho_t(x):=\int_0^1\chi_{(c_n h_t'(s))^{-1/n}}(x) h_t'(s)\,ds \quad\text{ for } t\in[0,1].
\end{equation}

Note that $\rho_t$ itself is not the linear interpolation of $\rho_0$ and $\rho_1$. In the next proposition, we will show that if $\rho_0, \rho_1$ are both continuous, compactly supported, and strictly radially decreasing within their supports (which is indeed the case for stationary solutions to \eqref{eq:evolution}), then $\{\rho_t\}_{t\in[0,1]}$ is a Lipschitz curve in $\mathcal{P}_2(\mathbb{R}^n)$ with respect to the 2-Wasserstein distance, although it is not a geodesic. In addition, the energy functional is continuous in $t$ for $t\in[0,1]$.
\begin{proposition}\label{prop_lipschitz}
Let $\rho_0, \rho_1\in \mathcal{P}(\R^n) \cap C(\mathbb{R}^n)$. Assume both of them are compactly supported, and strictly radially decreasing within the support of each of them. Consider the interpolation curve $\{\rho_t\}_{t\in[0,1]}$ defined by \eqref{h_t} and \eqref{rho_t}. Then $\{\rho_t\}_{t\in[0,1]}$ is a Lipschitz curve in $\mathcal{P}_2(\mathbb{R}^n)$ with respect to the 2-Wasserstein distance $d_2$.

In addition, if $m\geq 1$ and $W$ satisfies \textup{(W1)} and \textup{(W2)}, we have\begin{equation}\label{eq:continuous}
\mathcal{E}[\rho_t]\;\;\mbox{is continuous in $t$ for $t\in[0,1]$.}    
\end{equation}
\end{proposition}

The proof of Proposition~\ref{prop_lipschitz} will be postponed to Section~\ref{subsec_lipschitz}.
The main motivation for us to define this interpolation curve $\rho_t$ is that it turns out that the interaction energy $\mathcal{I}[\rho_t]$ is strictly convex along this curve for \emph{all attractive potentials}.  In addition, the entropy $\mathcal{S}[\rho_t]$ is convex if and only if $m\geq 2$. As a result, we have that $\mathcal{E}[\rho_t]$ is strictly convex along this curve for all attractive potentials when $m\geq 2$. The convexity of $\mathcal{S}[\rho_t]$ for $m\geq 2$ is rather straightforward to show, and we present the proof here.

\begin{proposition}\label{prop_convex_s}
Let $\rho_0, \rho_1 \in  \mathcal{P}(\mathbb{R}^n) \cap L^\infty(\R^n)$  be  both radially decreasing. Consider the interpolation curve $\{\rho_t\}_{t\in[0,1]}$ as given in \eqref{h_t} and \eqref{rho_t}. Then for $m\geq 2$ the function $t\mapsto \mathcal{S}[\rho_t]$ is convex for $t\in (0,1)$. 
\end{proposition}

\begin{proof}
Let $\Phi:\R^+\to \R^+$ be a strictly increasing function that is smooth in $(0,\infty)$, and satisfies  $\Phi'(a)>0$ for all $a>0$. Then we have
\begin{equation}\label{Internal Energy}
\begin{split}
\ds \int_{\R^n}\Phi(\rho_t(x))\,dx&= \ds\int_0^\infty |\{x \in \mathbb{R}^n: \Phi(\rho_t(x))\ge h\}|\,dh\\
&=\ds\int_0^\infty |\{x \in \mathbb{R}^n:  \rho_t(x)\ge \Phi^{-1}(y)\}|\,dy\\
&=\ds\int_0^1 |\{x: \rho_t(x)\ge h_t(s)\}|\, \Phi'(h_t(s))\, h_t'(s)\,ds\\
&=\int_0^1 \Phi'(h_t(s)) \,ds,
\end{split}
\end{equation}
where the third equality follows from the change of variable $y=\Phi(h_t(s))$, and the last equality is due to  \eqref{eq_h'}.

In general, the convexity of this integral in $t$ requires that $\Phi'$ being convex. To see this, taking the second derivative in $t$ of \eqref{Internal Energy}, and using the definition of $h_t(s)$ in \eqref{h_t}, we have
\[
\frac{d^2}{dt^2} \int_{\R^n}\Phi(\rho_t(x))\,dx = \int_0^1 \Phi^{(3)} (h_t(s)) \,(h_1(s)-h_0(s))^2 \,ds.
\]
For the case 
$$
\Phi_m(a):=\begin{cases}
\displaystyle\frac{a^{m}}{m-1}&\text{ for }m\ne 1,\\[0.2cm]
a\log a &\text{ for }m=1,
\end{cases}
$$
we have
\[
\frac{d^2}{dt^2} \mathcal{S}[\rho_t] = \frac{d^2}{dt^2} \int_{\R^n}\Phi_m(\rho_t(x))\,dx = m(m-2) \int_0^1 h_t(s)^{m-3} \,(h_1(s)-h_0(s))^2\, ds,
\]
which is zero for $m=2$, strictly positive for $m>2$, and strictly negative for $1\leq m<2$. This finishes the proof.
\end{proof}

Next we prove that the interaction energy is strictly convex along the interpolation curve in one dimension. The convexity proof in multi-dimension is computationally involved, thus we postpone it to Section~\ref{sec_convexity}. 

\begin{proposition}\label{prop_convex_1d}Assume $W$ satisfies \textup{(W1)} and \textup{(W2)}.
 Let $\rho_0, \rho_1\in \mathcal{P}(\R) \cap C(\mathbb{R})$ be two symmetric decreasing probability densities on $\mathbb{R}$ that are not identical. Consider the interpolation curve $\{\rho_t\}_{t\in[0,1]}$ as given in \eqref{h_t} and \eqref{rho_t}. Then the function $t\mapsto \mathcal{I}[\rho_t]$ is strictly convex for $t\in (0,1)$. 
\end{proposition}

\begin{proof} Let us first prove the convexity of $t\mapsto \mathcal{I}[\rho_t]$, and we will upgrade it to strict convexity at the end of the proof. 
Note that the integral $\int_{\mathbb{R}} \rho (W*\rho) dx$ is linear with respect to $W$. Since $W$ is a radial function with $W'(r)>0$ for $r>0$, it suffices to prove convexity of $t\mapsto \mathcal{I}[\rho_t]$ for each interaction potential of the form
\begin{equation}\label{def_wa}
W_a(r) := \begin{cases}
0 & 0\leq r<a\\
1 & r\geq a,
\end{cases}
\end{equation}
where $a>0$.
To see this, note that once we prove the convexity of $\mathcal{I}[\rho_t]$ for each $W_a$, for any attractive interaction potential $W$ that is bounded below, we can express $W$ (in the radial variable $r$) as
\begin{equation}\label{eq:decompW}
W(r) = \int_0^\infty W'(a) W_a(r) da + w_0,
\end{equation} 
where $w_0 = \lim_{r\to 0^+}W(r)$. Convexity of the function $t\mapsto \mathcal{I}[\rho_t]$ immediately follows from the linearity of $\mathcal{I}$ in $W$ and the fact that $W'(a)>0$ for $a>0$. And if $W$ is unbounded below, we can set $W_\epsilon := \max\{W, -\epsilon^{-1}\}$, which is bounded for each $\epsilon>0$. Since the interaction energy with potential $W_\epsilon$ is convex for all $\epsilon>0$, sending $\epsilon\to 0$ gives the convexity of $t\mapsto \mathcal{I}[\rho_t]$.

From now on, we replace $W$ by $W_a$ in the definition of $\mathcal{I}$. Using \eqref{h_rho}, $\mathcal{I}[\rho_t]$ can be written as
\begin{equation*}
\begin{split}
\mathcal{I}[\rho_t] &= \frac{1}{2}\int_0^1 \int_0^1 h_t'(s_1) h_t'(s_2) \left|\Big\{(x,y): |x|\leq \frac{1}{2h_t'(s_1)}, |y|\leq \frac{1}{2h_t'(s_2)}, |x-y|\geq a\Big\}\right| ds_1 ds_2.
\end{split}
\end{equation*}
Let us denote the integrand by $I(t; s_1, s_2)$. In the rest of the proof, we aim to show that $t\mapsto I(t; s_1, s_2)$ is convex in $(0,1)$ for a.e. $s_1, s_2 \in (0,1)$, which would directly imply the convexity of $t\mapsto \mathcal{I}[\rho_t]$.

For any $s_1, s_2 \in (0,1)$, we denote by $R(t; s_1, s_2)$ the rectangle centered at (0,0), with width and height given by $\frac{1}{h_t'(s_1)}$ and $\frac{1}{h_t'(s_2)}$ respectively. Note that $I(t; s_1, s_2) \in [0,1)$ outputs the portion of the rectangle lying outside the diagonal stripe $S_a := \{(x,y)\in \mathbb{R}^2:|x-y|< a\}$, that is,
\[
I(t; s_1, s_2) = \frac{|R(t; s_1, s_2) \cap (S_a)^c|}{|R(t; s_1, s_2)|}.
\]

For every $t, s_1, s_2 \in (0,1)$, depending on the side lengths of $R(t; s_1, s_2)$,  exactly one of the following four cases can happen. See Figure \ref{fig_3cases} for an illustration of Cases 1--3.

Case 0. Some vertices of $R(t; s_1, s_2)$ fall on $\partial S_a$.

Case 1. All four vertices of $R(t; s_1, s_2)$ belong to $S_a$.

Case 2. All four vertices of $R(t; s_1, s_2)$ belong to $(\overline{S_a})^c$.

Case 3. Two vertices of $R(t; s_1, s_2)$ belong to $S_a$, and the other two are in $(\overline{S_a})^c$.

\begin{figure}[h!]
\begin{center}
\includegraphics[scale=0.9]{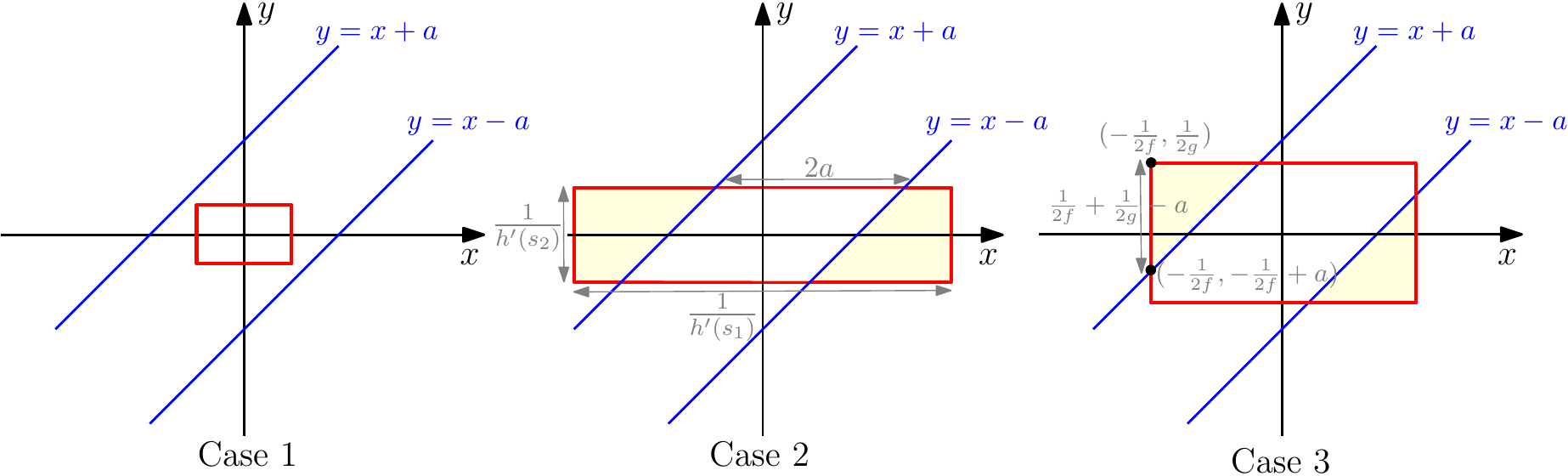}
\end{center}
\caption{Illustration of the three cases for the rectangle $R(t; s_1, s_2)$.\label{fig_3cases}}
\end{figure}

We start by pointing out that for every $t\in(0,1)$, Case 0 only happens for a zero-measure set of $(s_1, s_2) \in (0,1)^2$, therefore can be omitted. This is because for every $t\in(0,1)$, the function $s\mapsto h_t'(s)$ is strictly increasing for $s\in (0,1)$, which follows from Lemma~\ref{lemma_h}(a) and the fact that $\rho_0$, $\rho_1$ are both continuous. As a result, for every $t\in (0,1)$ and $s_1\in (0,1)$, there are at most two values of $s_2\in(0,1)$ such that Case 0 happens, which yields a zero-measure subset in $(0,1)^2$.

In the rest of the three cases, we aim to show that $\frac{\partial^2}{\partial t^2} I(t; s_1, s_2) \geq 0$. Case 1 is straightforward: first note that in this case we have $R(t; s_1, s_2) \subset S_a$, thus $I(t; s_1, s_2) = 0$. In addition, the continuity of the map $t\mapsto h_t'(s_i)$ for $i=1,2$ gives that $R(\tilde t; s_1, s_2) \subset S_a$ for all $\tilde t$ sufficiently close to $t$, leading to $\frac{\partial^2}{\partial t^2} I(t; s_1, s_2) = 0$.

In Case 2, without loss of generality we assume that $\frac{1}{h_t'(s_1)} \geq \frac{1}{h_t'(s_2)}$. (That is, the width is longer than height). A direct computation (see Figure \ref{fig_3cases} for an illustration) yields that
\[
|R(t; s_1, s_2) \cap (S_a)^c| = \frac{1}{h_t'(s_2)} \left( \frac{1}{h_t'(s_1)}-2a\right), 
\]
thus
\[
I(t; s_1, s_2) = \frac{|R(t; s_1, s_2) \cap (S_a)^c|}{|R(t; s_1, s_2)|} = 1-2ah_t'(s_1).
\]
Again, the continuity of the map $t\mapsto h_t'(s)$ gives that $\tilde t, s_1, s_2$ belongs to Case 2 for all $\tilde t$ sufficiently close to $t$, thus
\[
\frac{\partial^2}{\partial t^2} I(t; s_1, s_2)  = -2a \frac{\partial^2}{\partial t^2} h_t'(s_1) = 0,
\]
where in the last equality we used that $t\mapsto h_t'(s_1)$ is an affine function in $t$.

In Case 3, let us denote $f(t) := h_t'(s_1), g(t) := h_t'(s_2)$.  We then have that the half-width and half-height of $R(t;s_1, s_2)$ are $1/(2f)$ and $1/(2g)$ respectively. Note that $t, s_1, s_2$ belong to Case 3 if and only if
\begin{equation}\label{L_condition}
\left|\frac{1}{2f} - \frac{1}{2g}\right| < a \quad\text{ and }\quad \frac{1}{2f}+\frac{1}{2g}>a.
\end{equation}
In this case, the set $R(t; s_1, s_2) \cap (S_a)^c$ consists of two identical isosceles right triangles, whose legs have length $\frac{1}{2f} + \frac{1}{2g} - a$. (See Figure \ref{fig_3cases} for an illustration of this fact.) We thus have
\[
I(t; s_1, s_2) = \frac{|R(t; s_1, s_2) \cap (S_a)^c|}{|R(t; s_1, s_2)|} = \frac{(\frac{1}{2f} + \frac{1}{2g} - a)^2 }{(fg)^{-1}} = \frac{f}{4g} + \frac{g}{4f} + a^2 fg -af - ag + \frac{1}{2}.
\]
Recall that $f, g$ are positive for all $t\in [0,1]$, and they are affine functions of $t$. Thus $f'(t), g'(t)$ are constants, which may be positive, negative, or zero. Taking the second derivative in $t$ gives us
\begin{equation*}
\begin{split}
\frac{\partial^2}{\partial t^2} I(t; s_1, s_2) &=  \frac{g (f')^2}{2f^3} + \frac{f (g')^2}{2g^3} +  \left(2a^2 - \frac{1}{2f^2} - \frac{1}{2g^2}\right) f'g', \\
\end{split}
\end{equation*}
where the right hand side is a quadratic form of $f'$ and $g'$. We compute its discriminant as
\[
\begin{split}
\Delta &:= \left(2a^2 - \frac{1}{2f^2} - \frac{1}{2g^2}\right)^2 - 4\cdot \frac{g}{2f^3}\cdot \frac{f}{2g^3}\\
&= 4\left(a-\frac{1}{2f}+\frac{1}{2g}\right)\left(a+\frac{1}{2f}-\frac{1}{2g}\right)\left(a+\frac{1}{2f}+\frac{1}{2g}\right) \left(a-\frac{1}{2f}-\frac{1}{2g}\right) < 0,
\end{split}
\]
where the inequality is due to the follow reasoning: among the four parentheses on the right hand side,  the first two are both positive due to the first inequality of \eqref{L_condition}, the third is positive due to $f,g> 0$, but the fourth is negative due to the second inequality of \eqref{L_condition}.  Combining $\Delta<0$ with the fact that $\frac{g}{2f^3}, \frac{f}{2g^3}>0$, we have that $\frac{\partial^2}{\partial t^2} I(t; s_1, s_2) \geq 0$ for any $f', g'\in \mathbb{R}$, and in fact is strictly positive as long as $f', g'$ are both non-zero.

%
%
%
This finishes the convexity proof for Case 1--Case 3. Note that $t\mapsto I(t;s_1,s_2)$ is $C^1$ in $[0,1]$ (clearly the denominator is smooth for $t\in[0,1]$, and one can easily check that the numerator $|R(t; s_1, s_2) \cap (S_a)^c|$ is $C^1$ in $[0,1]$), and is piecewise smooth. In Case 1--Case 3 we have shown that $I$ is convex on each piece, leading to the convexity of $t\mapsto I(t; s_1, s_2)$ for $t\in[0,1]$ for a.e. $s_1, s_2 \in (0,1)$. Thus the function $t\mapsto \mathcal{I}[\rho_t]$ is convex for $t\in (0,1)$.

Finally, to improve the convexity  of $t\mapsto \mathcal{I}[\rho_t]$ into strict convexity, note that if $\rho_0$ and $\rho_1$ are not identical, $h_0'(s)$ and $h_1'(s)$ cannot be identical in $(0,1)$, thus they must be strictly ordered in some open interval. Without loss of generality, assume that there is some $s_0\in (0,1)$, such that $h_0'(s) < h_1'(s)$ in some small open neighborhood of $s_0$. Then for every $a\in (0,h'_0(s_0)/2)$, we have that $R(t; s_1, s_2)$ belong to Case 3 for all $s_1, s_2$ sufficiently close to $s_0$, and for these $s_1, s_2$ we have $\partial_t h_t'(s) = h_1'(s)-h_0'(s) > 0$ for $t\in (0,1)$, implying that $\partial_t h_t'(s_1) \partial_t h_t'(s_2)>0$ for $t\in (0,1)$. Thus $\frac{\partial^2}{\partial t^2} I(t; s_1, s_2)>0$ for a positive measure of $(s_1, s_2)$ for all sufficiently small $a>0$, implying the strict convexity of $t\mapsto \mathcal{I}[\rho_t]$.
\end{proof}

\begin{remark}
Even though in \textup{(W1)} we assume that $W'(r)>0$ for all $r>0$, from the last paragraph of the proof of Proposition~\ref{prop_convex_1d}, one can see that we can obtain strict convexity under a weaker assumption: all we need is that $W(r)$ is non-decreasing in $r$ for $r>0$, and $W'(r)>0$ in $(0,r_0)$ for some $r_0>0$. The same result holds for the multi-dimension proof in Proposition~\ref{prop_convex_nd}.
\end{remark}

The multi-dimension proof will be postponed to Proposition~\ref{prop_convex_nd}. 
Finally, recall that $\mathcal{E} = \mathcal{S} + \mathcal{I}$.  Proposition~\ref{prop_convex_s} gives the  convexity of $\mathcal{S}[\rho_t]$ along the interpolation curve for $m\geq 2$, whereas Proposition~\ref{prop_convex_1d} (1D case) and Proposition~\ref{prop_convex_nd} (multi-dimension case) give the strict convexity of $\mathcal{I}[\rho_t]$ along the interpolation curve. Combining these results together immediately leads to the strict convexity of $\mathcal{E}[\rho_t]$ for $m\geq 2$.

\begin{theorem}\label{thm_convex} Let $\rho_0, \rho_1\in \mathcal{P}(\R^n) \cap C(\mathbb{R}^n)$ be two radially decreasing probability densities on $\mathbb{R}^n$ that are not identical. Consider the interpolation curve $\{\rho_t\}_{t\in[0,1]}$ as given in \eqref{h_t} and \eqref{rho_t}. Then the function $t\mapsto \mathcal{E}[\rho_t]$ is strictly convex for $t\in (0,1)$ for $m\geq 2$. 
\end{theorem}

\section{Definitions and uniqueness proof of steady states}\label{sec:definitions}

In this section, we state the notion of steady states to \eqref{eq:evolution} and their properties. We then present the proofs for uniqueness of steady states. We first give a shorter proof when \eqref{eq:evolution} has a rigorous gradient flow structure in Section~\ref{sec_flow}. We then deal with the general case in Section~\ref{sec_general}. 

\subsection{Definition and basic properties of steady states}
We define the steady state of \eqref{eq:evolution} as follows, similar to \cite[Definition 1]{CHMV}.
\begin{definition}\label{def:gendef}
    We say that $\rho_s\in \mathcal{P}(\R^n) \cap L^\infty(\mathbb{R}^n)$ is a steady state of the evolution equation \eqref{eq:evolution}, if $\nabla \rho_s^m\in H^1_{loc}(\R^n)$, $\nabla W*\rho_s\in L^1_{loc}(\R^n)$ and
    \begin{equation}\label{critp}
                \nabla \rho_s^m=-\rho_s\nabla W*\rho_s \quad\text{ in }\mathbb{R}^n,
    \end{equation}
    in the sense of distributions in $\mathbb{R}^n$. 
    
    If $W$ satisfies \textup{(W2)} for some $k \in (-n, 1-n)$, we further require $\rho_s \in C^{0,\alpha}(\mathbb{R}^n)$ for some $\alpha \in (1-k-n, 1)$. If $\lim_{|x|\to\infty} W(x)\to\infty$, we further require  $\rho_s W(1+|x|) \in L^1(\mathbb{R}^n)$.
\end{definition}

Below we state some properties of the steady state that we will use later. The proof is a combination of the arguments in the previous literature, which we briefly outline for the sake of completeness. 

\begin{lemma}\label{lemma_stat}
Assume $m>1$, and $W$ satisfies \textup{(W1)--(W4)}. Let $\rho_s \in \mathcal{P}(\R^n) \cap L^\infty(\mathbb{R}^n)$ be a steady state to \eqref{eq:evolution} in the sense of Definition~\ref{def:gendef}. Then $\rho_s$ satisfies the following:

\begin{enumerate}
\item[\textup{(a)}] $W * \rho_s \in L^\infty_{loc}(\mathbb{R}^n)$, and $\nabla W*\rho_s \in L^\infty(\mathbb{R}^n)$.
\item[\textup{(b)}]  $\rho_s \in C(\mathbb{R}^n)$, and it is radially decreasing up to a translation.
\item[\textup{(c)}] $\frac{m}{m-1}\rho_s^{m-1} + W*\rho_s = C$ in $\supp\rho_s$ for some constant $C$.
\item[\textup{(d)}]  $\rho_s$ is compactly supported if either $\lim_{|x|\to\infty} W(x)=+\infty$, or $m\geq 2$. 
\item[\textup{(e)}] After a translation, denote $\supp \rho_s = B(0,R_s)$. Then we have $\rho_s$ is strictly decreasing in $r$ for $r\in (0,R_s)$.
\end{enumerate}
\end{lemma}

\begin{proof} 
The proof of $W * \rho_s \in L^\infty_{loc}(\mathbb{R}^n)$ can be done in the same way as \cite[equation (2.4)]{CHVY}, where we used (W4) and the assumption that $\rho_s W(1+|x|) \in L^1(\mathbb{R}^n)$. To prove $\nabla W*\rho_s \in L^\infty(\mathbb{R}^n)$, if $W$ satisfies (W2) with $k>1-n$, we directly decompose the convolution integral $\nabla W*\rho$ into near- and far-field sets $\mathcal{A}:=\{y:|x-y|<1\}$ and $\mathcal{B}:=\{y:|x-y|\geq1\}$, and use (W2) and (W3) to control these two integrals respectively. For $k \in (-n, 1-n)$, we use the additional H\"older regularity of $\rho_s$ in Definition \ref{def:gendef} and proceed as in the proof of \cite[Lemma~2.2(ii)]{CCH17}.

Once (a) is obtained, a standard argument (see the proof of \cite[Lemma 2.3]{CHVY} for example) gives 
\begin{equation}\label{eq_stat_temp}
\frac{m}{m-1}\rho_s^{m-1} +  W*\rho_s = C_i \quad\text{ in } \supp\rho_s,
\end{equation}
 where $C_i$ can be different if $\rho_s$ has more than one connected components.

The proof of (b) follows from \cite[Theorem 3]{CHMV}: even though the theorem is stated for Riesz potentials, it does not use any special homogeneity property of the potential, and the proof can be directly adapted to potentials satisfying (W1), (W2) and (W3). Note that $\rho_s$ being radially decreasing implies that its support has a single connected component, and combining this with \eqref{eq_stat_temp} gives (c).

To show (d), if $\lim_{r\to\infty} W(r) = +\infty$, it implies $\lim_{r\to\infty} (W*\rho_s)(r)=+\infty$ for any radially decreasing $\rho_s\in\mathcal{P}(\mathbb{R}^n)$. Thus $\rho_s$ must have compact support, otherwise it would violate (c). Next we aim to prove (d) under the condition $m\geq 2$ and $\lim_{r\to\infty} W(r)$ being finite. Without loss of generality let $\lim_{r\to\infty} W(r)=0$, so that $W$ is non-positive by assumption (W1). Towards a contradiction, assume that $\supp \rho_s = \mathbb{R}^n$. The fact that $\rho_s\in \mathcal{P}(\mathbb{R}^n)$ is radially decreasing gives $\lim_{r\to\infty} \rho_s(r)= 0$ and $\lim_{r\to\infty} (W*\rho_s)(r)=0$, so (c) becomes 
\begin{equation}\label{temp_rhos}
\frac{m}{m-1}\rho_s^{m-1} =(-W)*\rho_s \quad\text{ in } \mathbb{R}^n.
\end{equation}
If $m>2$, we have $\rho_s^{m-1}(r) \ll \rho_s(r)$ as $r\to\infty$, but on the other hand because $\rho$ is radially symmetric decreasing and $-W$ is positive there exists some $c(W)>0$ such that $((-W)*\rho_s)(r)\geq c\rho_s(r)$ for all $r>1$, and these two facts contradict with \eqref{temp_rhos}. If $m=2$, note that \eqref{temp_rhos} and the fact that $(-W)\geq 0$ imply that $W\in L^1(\mathbb{R}^n)$ and $\int_{\mathbb{R}^n} (-W)dx= 2$, which allows us to take the Fourier transform on both sides of \eqref{temp_rhos} and obtain
\[
2 \hat \rho_s(\xi) = - \hat{W}(\xi) \hat \rho_s(\xi) \quad\text{ for all }\xi \in \mathbb{R}^n.
\]
We can then apply the proof of \cite[Theorem 3.5]{BDF} to conclude that there cannot be such a $W$.

Finally, to prove (e), all we need is to improve the radially decreasing result of $\rho_s$ in (b) into strictly radially decreasing within $\supp\rho_s$. By (c), we have $\frac{m}{m-1}\rho_s^{m-1} = C-W*\rho_s$ for some $C$ in $B(0,R_s)$. Using the fact that $\rho_s$ is radially decreasing and $-W$ is strictly radially increasing (due to (W1)), we have that $C-W*\rho_s$ is strictly radially decreasing, which implies that $\rho_s = (\frac{m-1}{m}(C-W*\rho_s))^{1/(m-1)}$ is also strictly radially decreasing for $r\in(0,R_s)$.
\end{proof}

\subsection{A shortcut of uniqueness proof for equations with a gradient flow structure}\label{sec_flow}
As we have discussed in the discussion, \eqref{eq:evolution} is formally a gradient flow of $\mathcal{E}$ in $\mathcal{P}_2(\mathbb{R}^n)$ endowed with the 2-Wasserstein distance $d_2$. The gradient flow theory has been rigorously established for $\lambda$-convex potential $W$; see the books \cite{AGS, Villani} and the references therein. For such potentials, for any $\rho_0 \in \mathcal{P}_2(\mathbb{R}^n)$ with $\mathcal{E}[\rho_0]<\infty$, there exists a unique gradient flow $\rho(t)$ of $\mathcal{E}[\rho]$  in $\mathcal{P}_2(\mathbb{R}^n)$ endowed with the 2-Wasserstein distance $d_2$. In addition, the gradient flow coincides with the unique weak solution of \eqref{eq:evolution} with initial data $\rho_0$. Recently, the gradient flow theory has been generalized to energy functionals $\mathcal{E}$ that are $\omega$-convex \cite{CraigNonconvex} (where $\omega$ is some modulus of convexity), and to the attractive Newtonian potential \cite{CS18}. 

For all these potentials where the gradient flow theory has been rigorously established, we have a short proof of uniqueness of steady states, which we present below.

\begin{theorem}\label{thm_gf}
Assume $m\geq 2$, and $W$ satisfies \textup{(W1)--(W4)}. In addition, assume $W$ is such that \eqref{eq:evolution} has a local-in-time unique gradient flow solution for any $\rho_0 \in \mathcal{P}_2(\mathbb{R}^n)$ with $\mathcal{E}[\rho_0]<\infty$. Then \eqref{eq:evolution} has at most one steady state in  $L^\infty(\mathbb{R}^n) \cap \mathcal{P}_2(\mathbb{R}^n)$ up to a translation.
\end{theorem}

\begin{proof}
Towards a contradiction, assume that there exist two steady states $\rho_0^s$, $\rho_1^s$  in  $L^\infty(\mathbb{R}^n) \cap \mathcal{P}_2(\mathbb{R}^n)$ that are not identical to each other up to a translation. By Lemma~\ref{lemma_stat}(b,d,e), both $\rho_0^s$ and $\rho_1^s$ are continuous, compactly supported (thus $\mathcal{E}[\rho_0^s], \mathcal{E}[\rho_1^s] < \infty$), and strictly radially decreasing (within its support) up to a translation. Without loss of generality we can assume both $\rho_0^s$ and $\rho_1^s$ are both centered at the origin. For $t\in [0,1]$, let $\rho_t$ be the interpolation curve connecting $\rho_0^s$ and $\rho_1^s$ given by \eqref{h_t} and \eqref{rho_t}.

Next we consider the energy functional $\mathcal{E}[\rho_t]$ along the interpolation curve, where $\mathcal{E}[\rho_t]$ is continuous for $t\in[0,1]$ by Proposition~\ref{prop_lipschitz}, and strictly convex by Theorem~\ref{thm_convex}. Note that for a continuous and convex function in $[0,1]$, its left/right derivative is well-defined pointwise (which may be $\pm\infty$ at the endpoints), and the strict convexity gives that 
\[
\frac{d^+}{dt} \mathcal{E}[\rho_t]\Big|_{t=0} < \frac{d^-}{dt} \mathcal{E}[\rho_t] \Big|_{t=1}.
\] Thus
at least one of the following two conditions must be true :
\begin{equation}\label{two_eq}
-\frac{d^+}{dt} \mathcal{E}[\rho_t]\Big|_{t=0}=\lim_{t\to 0^+} \frac{\mathcal{E}[\rho_0^s] - \mathcal{E}[\rho_t]}{t} \in (0,+\infty] \quad\text{ or }\quad \frac{d^-}{dt} \mathcal{E}[\rho_t]\Big|_{t=1} = \lim_{t\to 1^-} \frac{\mathcal{E}[\rho_1^s] - \mathcal{E}[\rho_t]}{1-t}  \in (0,+\infty].
\end{equation}
In the rest of the proof, using the gradient flow structure and the fact that both $\rho_0^s$ and $\rho_1^s$ are steady states, we will show that both inequalities in \eqref{two_eq} must be false, 
leading to a contradiction.

By assumption, for any initial condition $\rho_0 \in \mathcal{P}_2(\mathbb{R}^n)$ with $\mathcal{E}[\rho_0]<\infty$ and any $T>0$, there exists an absolutely continuous curve $\rho:[0,T)\to \mathcal{P}_2(\R^n)$ with respect to 2-Wasserstein distance $d_2$, which is a gradient flow of $\mathcal{E}$ and corresponds to the unique distributional solution to \eqref{eq:evolution}. More specifically, $\rho(t)$ satisfies an Evolution Variational Inequality (EVI): see \cite[Definition 3.5]{AmGi} for $\lambda$-convex potential, and \cite[Definition 2.10]{CraigNonconvex} when $W$ is the Newtonian potential. Then arguing as in \cite[Proposition 3.6]{AmGi} (see also \cite{CarrilloLisiniMainini} when $W$ is the Newtonian potential), we have that the EVI implies the following Energy Dissipation Inequality (EDI):
$$
\mathcal{E}[\rho(t)]+\frac{1}{2}\int_0^t |\partial \mathcal{E}[\rho(\tau)]|^2\,d\tau+\frac{1}{2}\int_0^t |\dot\rho(\tau)|^2\,d\tau\le \mathcal{E}[\rho(0)]\qquad\mbox{for all $t\in(0,T)$},
$$
where the metric slope $|\partial \mathcal{E}|$ and the metric derivative of $\rho$ at time $\tau$ are given respectively by
$$
|\partial \mathcal{E}[\rho(\tau)]|= \limsup_{\nu\to \rho(\tau)} \frac{(\mathcal{E}[\rho(\tau)]-\mathcal{E}[\nu])_+}{d_2(\rho(\tau), \nu)}\qquad\mbox{and}\qquad |\dot\rho(\tau)|=\limsup_{h\to 0} \frac{d_2(\rho(\tau+h),\rho(\tau))}{h}.
$$
In particular, when the initial data is $\rho(0) = \rho_0^s$, $\rho(t) \equiv \rho_0^s$ is a distributional solution to \eqref{eq:evolution} due to its stationarity. Thus the EDI becomes
\[
\frac{t}{2} |\partial \mathcal{E}[\rho_0^s]|^2 \le 0 \quad\text{ for all }t\in (0,T),
\]
implying that  $|\partial \mathcal{E}[\rho_0^s]| = 0$.  Thus 
\begin{equation}\label{ineq_slope}
0 = |\partial \mathcal{E}[\rho_0^s]| = \limsup_{\nu\to \rho_0^s} \frac{(\mathcal{E}[\rho_0^s]-\mathcal{E}[\nu])_+}{d_2(\rho_0^s, \nu)} \geq  \lim_{t\to 0^+}
\frac{(\mathcal{E}[\rho_0^s]-\mathcal{E}[\rho_t])_+}{d_2(\rho_0^s, \rho_t)} \geq \lim_{t\to 0^+}
\frac{(\mathcal{E}[\rho_0^s]-\mathcal{E}[\rho_t])_+}{Ct}.
\end{equation}
Here the last inequality comes from the Lipschitz property in Proposition~\ref{prop_lipschitz}, where the Lipschitz constant $C>0$ is some finite constant depending on $\rho_0^s$ and $\rho_1^s$. Note that \eqref{ineq_slope} contradicts with the first condition of \eqref{two_eq}.

Likewise, since $\rho(t) \equiv \rho_1^s$ is also a steady state, the EDI leads to $|\partial \mathcal{E}[\rho_1^s]| = 0$. An identical argument as \eqref{ineq_slope} then gives
\[
0 = |\partial \mathcal{E}[\rho_1^s]| = \limsup_{\nu\to \rho_1^s} \frac{(\mathcal{E}[\rho_1^s]-\mathcal{E}[\nu])_+}{d_2(\rho_1^s, \nu)} \geq  \lim_{t\to 1^-}
\frac{(\mathcal{E}[\rho_1^s]-\mathcal{E}[\rho_t])_+}{d_2(\rho_1^s, \rho_t)} \geq \lim_{t\to 1^-}
\frac{(\mathcal{E}[\rho_1^s]-\mathcal{E}[\rho_t])_+}{C(1-t)},
\]
contradicting with the second condition of \eqref{two_eq}. We have shown that both equations of \eqref{two_eq} must be false, thus there cannot be two different radially decreasing steady states $\rho_0^s$ and $\rho_1^s$ in $L^\infty(\mathbb{R}^n) \cap \mathcal{P}_2(\mathbb{R}^n)$, finishing the proof.
\end{proof}

\color{black}

\subsection{Uniqueness proof for general attractive potentials}\label{sec_general}
Next we aim to prove uniqueness of steady states for general attractive interaction potentials, for which the gradient flow definition is not well defined.

To begin with, we need to establish some extra regularity and non-degeneracy properties of $\rho_s$. In the next lemma, we will show that $\rho_s$ is smooth inside its support, and has a strictly negative Laplacian at the origin. The smoothness property has been shown in \cite{CHMV} when $W$ is a Riesz potential.

\begin{lemma}\label{prop:regularity}
Assume $m>1$, and $W$ satisfies \textup{(W1)--(W4)}.  Let $\rho_s \in \mathcal{P}(\R^n) \cap L^\infty(\mathbb{R}^n)$ be a steady state to \eqref{eq:evolution} in the sense of Definition~\ref{def:gendef}, with center of mass at the origin. Then $\rho_s$ is smooth in the interior of its support, and satisfies the following non-degeneracy condition at the origin:
\begin{equation}\label{eq:nondegneracyprop}
    \Delta\rho_s(0)<0.
\end{equation}
\end{lemma}

\begin{proof}
By Lemma~\ref{lemma_stat}(b,d), $\rho_s$ is radially decreasing up to a translation, and compactly supported. Thus if $\rho_s$ has the center of mass at the origin, it must be centered at the origin, with its support being some open ball $B_{R_0}$ with $R_0<\infty$. By Lemma~\ref{lemma_stat}(a), $W*\rho_s$ is Lipschitz, therefore Lemma~\ref{lemma_stat}(c) gives that $\rho_s^{m-1}$ is Lipschitz in $B_{R_0}$. The fact that $\rho_s$ is radially decreasing yields that $\rho_s \geq \rho_s(r)>0$ in $B_r$ for any $r<R_0$, thus $\rho_s$ is Lipschitz in $B_{r}$ for $r<R_0$ (where the Lipschitz constant may blow up as $r\nearrow R_0$.) As a result, we have $\rho_s \in C^\alpha_{loc}(B_{R_0})$ for any $\alpha \in (0,1)$. 

Next we will use an iterative argument to obtain higher regularity of $\rho_s$ in $B_{R_0}$. Assuming that $\rho_s \in C^\alpha_{loc}(B_{R_0})$ for some $\alpha>0$, in each iterative step our goal is to show 
\begin{equation}\label{temp_claim0}
\rho_s \in C^{\beta}_{loc}(B_{R_0}) \quad\text{ for any }\beta < \alpha + \min\{n +k,2\},
\end{equation} where $k>-n$ is the power in the assumption (W2) such that $W'(r)\leq C_w r^{k-1}$ for all $0<r<1$. 
(Here if $\beta>1$, the notation $C^\beta$ stands for $C^{k,s}$, where $k = \lfloor \beta \rfloor$ and $s=\beta-k$.)

For any $0<R_1<R_2<\infty$, let $\phi_{R_1,R_2}$ be a radially decreasing function satisfying
\begin{equation}\label{eq:phi}
    \phi_{R_1, R_2}\in C^\infty_c(\mathbb{R}^n),\quad 0\le\phi_{R_1, R_2}\le 1, \quad\phi_{R_1, R_2}(x)= \begin{cases} 1& \text{ in } B_{R_1}\\
    0 &  \text{ in } (B_{R_2})^c.
    \end{cases}
\end{equation}

%
%


For any $0<\epsilon\ll 1$, by Lemma~\ref{lemma_stat}(c) (where we decompose $W = W\phi_{\epsilon, 2\epsilon} + W(1-\phi_{\epsilon, 2\epsilon})$), we have
\begin{equation}\label{eq0000}
\frac{m}{m-1}\rho_s^{m-1}+(W\phi_{\epsilon, 2\epsilon}) *\rho_s =C - (W(1-\phi_{\epsilon, 2\epsilon})) *\rho_s \quad\text{ in } B_{R_0}.
\end{equation}
Note that RHS of \eqref{eq0000} is $C^\infty$ in $B_{R_0}$, since $W\in C^\infty(\R^n\setminus\{0\})$ implies that $W(1-\phi_{\epsilon, 2\epsilon}) \in C^\infty(\mathbb{R}^n)$. Next we take a closer look on the second term on the left hand side, and aim to show that 
\begin{equation}\label{temp_claim}
(W\phi_{\epsilon, 2\epsilon}) *\rho_s \in C^{\beta}_{loc}(B_{R_0-4\epsilon}) \quad\text{ for any }\beta < \alpha + \min\{n + k,2\}.
\end{equation}
Once this is done, \eqref{eq0000} implies that $\frac{m}{m-1}\rho_s^{m-1} \in C^\beta_{loc}(B_{R_0-4\epsilon})$ for $\beta$ as above. The fact that $\rho_s$ is bounded below by a positive constant in $B_{R_0-4\epsilon}$ then implies that $\rho_s \in C^{\beta}_{loc}(B_{R_0-4\epsilon})$. Since $0<\epsilon\ll 1$ can be made arbitrarily small, we obtain \eqref{temp_claim0}.

Next we prove \eqref{temp_claim}. Let us decompose $(W\phi_{\epsilon, 2\epsilon}) *\rho_s$ as
\[
\begin{split}
(W\phi_{\epsilon, 2\epsilon}) *\rho_s &= \underbrace{(W\phi_{\epsilon, 2\epsilon})}_{=:f_1} *\underbrace{(\rho_s \phi_{R_0-2\epsilon,R_0-\epsilon})}_{:= g_1} + (W\phi_{\epsilon, 2\epsilon}) *(\rho_s (1-\phi_{R_0-2\epsilon,R_0-\epsilon}))\\
& =: T_1 + T_2.
\end{split}
\]
We deal with the term $T_1$ first. By Lemma~\ref{lem:app} in the appendix, the assumptions (W1) and (W2) imply that $f_1 := W \phi_{\epsilon, 2\epsilon}$ satisfies  
\begin{equation}\label{reg_temp}
(-\Delta)^s f_1 \in L^1(\mathbb{R}^n) \quad\text{ for any }0<s<\min\left\{\frac{n+k}{2}, 1\right\}.
\end{equation}
In addition, the iteration assumption $\rho_s \in C^\alpha_{loc}(B_{R_0})$ and the fact that $ \phi_{R_0-2\epsilon,R_0-\epsilon}$ is supported in $B_{R_0-\epsilon}$ give that $g_1 := \rho_s \phi_{R_0-2\epsilon,R_0-\epsilon} \in C_c^\alpha(\mathbb{R}^n)$.  Thus properties of fractional Laplacian \cite{Silvestre} gives that for any $\alpha' \in (0,\alpha)$,
\[
\|(-\Delta)^{\frac{\alpha'}{2}} g_1\|_{L^\infty} \leq C(R_0)\|(-\Delta)^{\frac{\alpha'}{2}} g_1\|_{C^{\alpha-\alpha'}} \leq C(R_0) (\|g_1\|_{C^\alpha} + \|g_1\|_{L^\infty}) < \infty,
\]
where in the first inequality we use that $g_1$ is supported in $B_{R_0}$, and in the second inequality we apply \cite[Proposition 2.5--2.7]{Silvestre}. 
Combining Young's inequality with the above estimates for $f_1$ and $g_1$, we have that for any $\alpha' \in (0,\alpha)$ and $s$ as in \eqref{reg_temp},
$$
\|(-\Delta)^{\frac{\alpha'}{2} + s}T_1\|_{L^\infty} = \|(-\Delta)^{\frac{\alpha'}{2} + s} (f_1 * g_1)\|_{L^\infty}\le \|(-\Delta)^s f_1\|_{L^1}\|(-\Delta)^{\frac{\alpha'}{2}} g_1\|_{L^\infty}<\infty.
$$
Finally, since $T_1 \in L^\infty(\mathbb{R}^n)$ and $(-\Delta)^{\alpha'/2 + s}T_1 \in L^\infty(\mathbb{R}^n)$, \cite[Proposition 2.9]{Silvestre} gives that 
\[
T_1 \in C^\beta(\mathbb{R}^n) \text{ for any }\beta < \alpha'+2s,
\]
that is,
$T_1 \in C^\beta(\mathbb{R}^n)$ for any $\beta < \alpha + \min\{n+k, 2\}$.

For the term $T_2$, since $ \rho_s (1-\phi_{R_0-2\epsilon,R_0-\epsilon}) \equiv 0$ in $B_{R_0-2\epsilon}$, and $\supp (W\phi_{\epsilon, 2\epsilon}) \subset B_{2\epsilon}$, we have that $T_2 \equiv 0$ in $B_{R_0-4\epsilon}$ (thus is smooth in $B_{R_0-4\epsilon}$). Combining this with the above regularity for $T_1$ finishes the proof of \eqref{temp_claim}, thus gives \eqref{temp_claim0} since $\epsilon>0$ can be made arbitrarily small. Once we obtain \eqref{temp_claim0}, we can iterate its proof and improve the regularity in $B_{R_0}$ in each iteration, thus $\rho_s$ is smooth in the interior of $B_{R_0}$.

%
%

Next, we show the non-degeneracy property \eqref{eq:nondegneracyprop}. For $0<\epsilon\ll 1$, we decompose $W=W\phi_{\epsilon,2\epsilon}+W(1-\phi_{\epsilon,2\epsilon})$. Using that $W(1-\phi_{\epsilon,2\epsilon})$ is smooth and radially symmetric, as well as the fact that $\rho_s$ is supported in $B_{R_0}$, we have
\begin{equation}\label{eq_laplace}
\begin{split}
\Delta \big( (W(1-\phi_{\epsilon,2\epsilon}))*\rho_s\big)(0)&=\int_{\R^n}\Delta (W(1-\phi_{\epsilon,2\epsilon}))(-y)\,\rho_s(y)\,dy\\
&=-\int_{B_{R_0}}\big(\nabla (W(1-\phi_{\epsilon,2\epsilon}))(-y)\big)\cdot\nabla\rho_s(y)\,dy\\
&=- \displaystyle\int_0^{R_0} \partial_r(W(1-\phi_{\epsilon,2\epsilon}))(r)\; \partial_r \rho_s(r)\; \omega_n r^{n-1}dr\\
&=-\underbrace{ \displaystyle\int_0^{R_0} (\partial_rW)(1-\phi_{\epsilon,2\epsilon}) \, (\partial_r \rho_s)\, \omega_n r^{n-1}dr}_{=:I_1} + \underbrace{ \int_0^{R_0} W (\partial_r \phi_{\epsilon,2\epsilon}) \, (\partial_r \rho_s)\, \omega_n r^{n-1}dr}_{=:I_2}.\\
\end{split}
\end{equation}
Since $\rho_s$ is smooth near the origin with $\nabla\rho_s(0)=0$, there exists some finite constant $C$ such that
$$
\partial_r\rho_s(r)\le C r \qquad\mbox{for any $r\in (0,2\epsilon)$.}
$$
Using that $W \in L^1_{loc}(\mathbb{R}^n)$ and that $|\partial_r\phi_{\epsilon,2\epsilon}| \le C\epsilon^{-1} 1_{(\epsilon,2\epsilon)}$, we bound $I_2$ as
$$
\lim_{\epsilon\to0^+}I_2 \le \int_\epsilon^{2\epsilon} W(r) C\epsilon^{-1} \cdot C\epsilon\cdot \omega_n r^{n-1}dr \leq C\lim_{\epsilon\to0^+}\|W\|_{L^1(B_{2\epsilon})}=0.
$$
As for $I_1$, using $\partial_r W > 0$, $\partial_r \rho_s\le 0$ and the fact that $\partial_r \rho_s \not \equiv 0$ for $r\in(0,R_0)$, monotone convergence theorem gives
\begin{equation}\label{eq:radial}
\lim_{\epsilon\to0^+} I_1 =  \int_0^{R_0} (\partial_r W)(\partial_r \rho_s ) \omega_n r^{n-1} dr>0.    
\end{equation}
Putting the above limits for $I_1$ and $I_2$ together gives
\begin{equation}\label{temp_laplace2}
\lim_{\epsilon\to0^+} \Delta \big( (W(1-\phi_{\epsilon,2\epsilon}))*\rho_s\big)(0) < 0.
\end{equation}
Next, we use the $C^2$ regularity of $\rho_s$ near the origin and Young's inequality to bound
$$
|\Delta\big((W\phi_{\epsilon,2\epsilon})*\rho_s\big)(0)|\le \|W\|_{L^1(B_{2\epsilon})}\|\Delta\rho_s\|_{L^\infty(B_{2\epsilon})}.
$$
Because $W\in L^1_{loc}(\R^n)$, we obtain that
\begin{equation}\label{eq:limit}
\lim_{\epsilon\to 0^+}|\Delta\big((W\phi_{\epsilon,2\epsilon})*\rho_s\big)(0)|\le \lim_{\epsilon\to 0^+}\|W\|_{L^1(B_{2\epsilon})}\|\Delta\rho_s\|_{L^\infty(B_{2\epsilon})}=0.    
\end{equation}
Combining \eqref{temp_laplace2} and \eqref{eq:limit}, we obtain the desired non-degeneracy result \eqref{eq:nondegneracyprop}.
\end{proof}

The above regularity and non-degeneracy result allows us to obtain some further regularity properties of the interpolation curve in Section~\ref{sec:further_reg}. Using these properties, (in particular, using Proposition~\ref{prop:general}), we are now ready to prove Theorem~\ref{thm:uniqueness} under the assumptions (W1)--(W4), without using the gradient flow structure.
\begin{proof}[\textbf{\textup{Proof of Theorem~\ref{thm:uniqueness}}}]

Towards a contradiction, assume that there exist two steady states $\rho_0^s$, $\rho_1^s$  in  $L^\infty(\mathbb{R}^n) \cap \mathcal{P}(\mathbb{R}^n)$ that are not identical to each other up to a translation. By the same argument as in the first two paragraphs of the proof of Theorem~\ref{thm_gf},  (note that the gradient flow structure was not used in these paragraphs), at least one of the following inequalities must be true:
\begin{equation}\label{two_eq2}
\lim_{t\to 0^+} \frac{\mathcal{E}[\rho_0^s] - \mathcal{E}[\rho_t]}{t} \in (0,+\infty] \quad\text{ or }\quad \lim_{t\to 1^-} \frac{\mathcal{E}[\rho_1^s] - \mathcal{E}[\rho_t]}{1-t}  \in (0,+\infty] .
\end{equation}
By Proposition~\ref{prop:general}, we know the first equation in \eqref{two_eq2} must be false. Likewise, by considering the same interpolation curve with $\rho_0$ and $\rho_1$ interchanged, Proposition~\ref{prop:general} gives that the second equation in \eqref{two_eq2} must also be false, leading to a contradiction.
\end{proof}

\section{Regularity and convexity along the interpolation curve}

\subsection{Lipschitz property of the interpolation curve}

\label{subsec_lipschitz}
Assume that $\rho_0, \rho_1 \in \mathcal{P}(\mathbb{R}^n) \cap L^\infty(\R^n)$ are two radially decreasing compactly supported probability densities. In this subsection, we prove Proposition \ref{prop_lipschitz}, which gives that $\{\rho_t\}_{t\in[0,1]}$ is a Lipschitz curve with respect to the 2-Wasserstein distance. We start with a simple lemma that gives some additional properties of the interpolation curve $\rho_t$.
\begin{lemma}\label{lem_supp}
Let $\rho_0, \rho_1 \in \mathcal{P}(\R^n) \cap L^\infty(\mathbb{R}^n)$ be radially decreasing and compactly supported. Consider the interpolation curve $\{\rho_t\}_{t\in[0,1]}$ defined by \eqref{h_t} and \eqref{rho_t}. Then we have the following:
\begin{itemize}
\item[(a)] Denote $\supp \rho_i =: B_{R_i}$ for $i=0,1$. Then $\rho_t$ is compactly supported for all $t\in(0,1)$ with
\begin{equation}\label{rho_t_supp}
\supp \rho_t = B_{R_t} \quad\text{ for all }t\in (0,1),
\end{equation}
where 
\begin{equation}\label{def_Rt}
R_t := ((1-t) R_0^{-n} + t R_1^{-n})^{-1/n}\quad\text{ for }t\in (0,1).
\end{equation} In particular, note that $R_t \leq \max\{R_0, R_1\}$ for $t\in [0,1]$.

\smallskip

\item[(b)]
If $\rho_0$ and $\rho_1$ are both strictly radially decreasing within their supports, then so is $\rho_t$ for all $t\in(0,1)$.

\smallskip
\item[(c)]
If $\rho_0$ and $\rho_1$ are both in $C(\mathbb{R}^n)$, then so is $\rho_t$ for all $t\in(0,1)$.

\end{itemize}
\end{lemma}
\begin{proof}Let us prove (a) first.
 Using the definition of $h_t$ in \eqref{h_t} and Lemma~\ref{lemma_h}(b), we have
\[
\begin{split}
\lim_{s\to0^+} h_t'(s) &= \lim_{s\to0^+} ((1-t) h_0'(s) + t h_1'(s))\\
& = (1-t) |\supp \rho_0|^{-1} + t |\supp \rho_1|^{-1}\\& = c_n^{-1}((1-t)R_0^{-n}+tR_1^{-n}),
\end{split}
\]
thus Lemma~\ref{lemma_h}(b) implies that
\[
|\supp \rho_t| = \left(\lim_{s\to0^+} h_t'(s)\right)^{-1} = c_n ((1-t)  R_0^{-n} + t R_1^{-n})^{-1} = c_n R_t^n,
\]
where $R_t$ is as defined in \eqref{def_Rt}.
Since $\rho_t$ is radially decreasing for all $t\in(0,1)$, the support of $\rho_t$ must be a ball centered at the origin, thus the above equality is equivalent with \eqref{rho_t_supp}.

Next we prove (b). Let us take any fixed $t\in(0,1)$. By Lemma~\ref{lemma_h}(d), we have that $h_i \in C^1((0,1))$ for $i=0,1$, thus it follows that $h_t \in C^1((0,1))$. By \eqref{eq_h_der}, we have that $|\{\rho>h_t(s)\}|$ is continuous in $s$ for $s\in(0,1)$. Combining with the fact that $h_t(s)$ is continuous and strictly increasing in $s$ (see Lemma~\ref{lemma_h}(a)), we have that $|\{\rho_t>h\}|$ is continuous in $h$ for $h\in(0,\|\rho_t\|_\infty)$, thus $\rho_t$ is strictly radially decreasing.

Finally, to prove (c), note that since both $\rho_0$ and $\rho_1$ are continuous, $h_0'(s)$ and $h_1'(s)$ must be both strictly increasing in $s$ in $(0,1)$ due to \eqref{eq_h_der}. As a result, for any $t\in(0,1)$, $h_t'(s)$ is also strictly increasing for $s\in(0,1)$. This shows that $\rho_t$ must be continuous: if $\rho_{t_0}$ is discontinuous for some $t_0\in(0,1)$, then $\rho_{t_0}(r)$ must have a jump discontinuity somewhere, thus $h_{t_0}'(s)$ must be a constant in some interval, a contradiction. 
\end{proof}

To understand the regularity properties of the interpolation curve, we will find a vector field $V:\mathbb{R}^n \times (0,1) \to \mathbb{R}^n$ such that
\begin{equation}\label{continuity}
\partial_t\rho_t(x)+\nabla \cdot(V(x,t) \rho_t(x))=0 \quad\text{ for  }t\in(0,1)\text{ and } x\in\mathbb{R}^n
\end{equation}
in the weak sense. Since $\rho_t$ is radial for $t\in[0,1]$, we can restrict ourselves to vector fields that have the form $V(x,t)=v(|x|,t)\frac{x}{|x|}$. Note that for a fixed $r>0$, $v(r,t)$ is related to $\frac{d}{dt} \int_{B_r} \rho_t(x)\,dx$ by the identity
\begin{equation}\label{41}
\frac{d}{dt} \int_{B_r} \rho_t(x)\,dx=-\int_{B_r} \nabla\cdot(V(x,t)\rho_t)\,dx= -\int_{\partial B_r}v(r,t)\rho_t(r)\,dS=-\omega_n r^{n-1}v(r,t)\rho_t(r).
\end{equation}
To determine $v(r,t)$, we will use another way to compute $\frac{d}{dt} \int_{B_r} \rho_t(x)\,dx$. \eqref{h_rho} yields that
\begin{equation}\label{temp_rho_r}
\begin{array}{rcl}
\ds \int_{B_r}\rho_t(x)\,dx&=&\ds \int_{B_r} \int_0^1\chi_{(c_n h_t'(s))^{-1/n}}(x) h_t'(s)\,ds dx\\[0.3cm]
&=&\ds \int_{0}^1\min\left\{c_n r^n,(h_t'(s))^{-1}\right\}  h_t'(s) \,ds.
\end{array}
\end{equation}
By Lemma~\ref{lemma_stat}(b,e) and Lemma~\ref{lem_supp}(b,c), for any $t\in[0,1]$ we have $\rho_t \in C(\mathbb{R}^n)$, and is strictly radially decreasing with its support. Thus for any $t\in[0,1]$ and $r \in (0, (c_n^{-1} |\supp \rho_t|)^{1/n})$, there exists a unique $s_{r,t}$ that satisfies the implicit equation
\begin{equation}\label{def_srt}
h_t'(s_{r,t})=(c_n r^n)^{-1}.
\end{equation}
 The definition of $s_{r,t}$ is illustrated in Figure~\ref{fig_srt}.
 By continuity of $\rho_t$, we know $s_{r,t}$ also satisfies
\begin{equation}\label{def_srt2}
\rho_t(r)=h_t(s_{r,t}).
\end{equation}

\begin{figure}[h!]
\begin{center}
\includegraphics[scale=0.9]{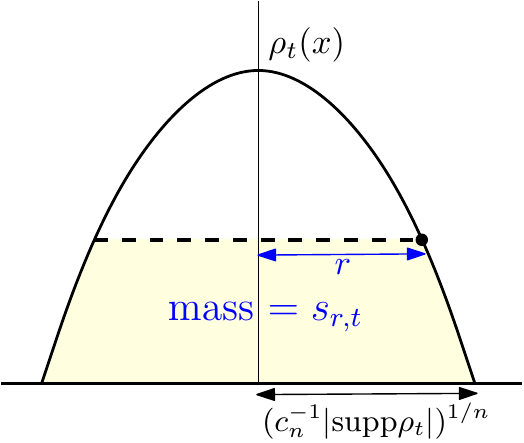}
\end{center}
\caption{Illustration of the definition of  $s_{r,t}$. \label{fig_srt}}
\end{figure}

Using \eqref{def_srt}, we can rewrite \eqref{temp_rho_r} as
\[
\int_{B_r}\rho_t(x)\,dx = \int_0^{s_{r,t}} c_n r^n h_t'(s) ds +  \int_{s_{r,t}}^1  ds = c_n r^n h_t(s_{r,t}) +(1-s_{r,t}).
\]

Differentiating in $t$, we obtain that
\[
\begin{split}
\frac{d}{dt}\int_{B_r} \rho_t(x)\,dx& =c_n r^n (\partial_t h_t)(s_{r,t}) + \underbrace{c_n r^n h'_t(s_{r,t})}_{=1} \partial_t s_{r,t} - \partial_t s_{r,t}\\
&= c_n r^n (h_1-h_0)(s_{r,t}).
\end{split}
\]
Combining this with \eqref{41} and \eqref{def_srt2} gives the following explicit expression of $v(r,t)$ (recall that $c_n = \omega_n / n$):
\begin{equation}\label{def_v}
v(r,t)=\frac{r (h_0-h_1)(s_{r,t})}{n \rho_t(r)}=\frac{r (h_0-h_1)(s_{r,t})}{nh_t(s_{r,t})} \quad\text{ for all }r \in (0, (c_n^{-1} |\supp \rho_t|)^{1/n}).
\end{equation}
With $v(r,t)$ explicitly given as above,  we are now ready to prove Proposition~\ref{prop_lipschitz}.

\begin{proof}[\textup{\textbf{Proof of Proposition~\ref{prop_lipschitz}}}]
From the above computation, the interpolation curve $\{\rho_t\}_{t\in[0,1]}$ satisfies the continuity equation \eqref{continuity} with $V(x,t) = v(|x|, t)\frac{x}{|x|}$, where $v(r,t)$ is given by \eqref{def_v}. We claim that such $V(x,t)$ satisfies the estimate
\begin{equation}\label{v_bd}
|V(x,t)| = v(|x|, t) \leq C|x| \quad \text{  for all } t\in (0,1) \text{ and } x\in \supp \rho_t,\end{equation} where $C<\infty$ depends only on $\rho_0$ and $\rho_1$, and is independent of $x$ and $t$.

In order to obtain \eqref{v_bd}, due to the explicit formula \eqref{def_v} for $v(r,t)$, it suffices to control $\frac{|(h_0-h_1)(s_{r,t})|}{h_t(s_{r,t})}$. For any $s\in(0,1)$, using that $h_t(s) \geq \min\{h_0(s), h_1(s)\}$, we have
\[
\frac{|(h_0-h_1)(s)|}{h_t(s)} \leq \frac{\max\{h_0(s),h_1(s)\}}{\min\{h_0(s),h_1(s)\}} \leq \max\left\{ \frac{h_0(s)}{h_1(s)}, \frac{h_1(s)}{h_0(s)}\right\}.
\]

For the fraction $ \frac{h_1(s)}{h_0(s)}$, by L'Hopital's rule and Lemma~\ref{lemma_h}(b), we have
$$
\lim_{s\to 0^+}\frac{h_1(s)}{h_0(s)}=\lim_{s\to 0^+}\frac{ h_1'(s)}{h_0'(s)}=\frac{|\{\rho_1>0\}|}{|\{\rho_0>0\}|} < \infty,
$$
where in the last inequality we used that $\rho_1$ has compact support. Also, 
\[
\lim_{s\to 1^-}\frac{h_1(s)}{h_0(s)} = \frac{\|\rho_1\|_\infty}{\|\rho_0\|_\infty} < \infty.
\] The continuity of $h_0$ and $h_1$ in $(0,1)$ then yields that $\sup_{s\in(0,1)} \frac{h_1}{h_0} < \infty$. An identical argument gives $\sup_{s\in(0,1)} \frac{h_0}{h_1} < \infty$. Thus 
\[
\sup_{s\in(0,1)} \frac{|(h_0-h_1)(s)|}{h_t(s)} < C
\] for some $C<\infty$ only depending on $\rho_0$ and $\rho_1$, finishing the proof of \eqref{v_bd}.

Once we obtain \eqref{v_bd},
using the Benamou-Brenier  representation of the 2-Wasserstein distance \cite[Proposition 1.1]{BenamouBrenier}, we have that for every $0\leq t_1<t_2\leq 1$, 
\begin{equation}\label{eq:Lipschitzproperty}
d_2(\rho_{t_1},\rho_{t_2})\le  \left(\sup_{t\in[t_1,t_2]} \int_{\R^n} |V(x,t)|^2\,d\rho_{t}(x)\right)^{1/2} |t_2-t_1| \leq C \max\{R_0, R_1\} |t_2 - t_1|,
\end{equation}
where the last inequality follows from \eqref{v_bd} and Lemma~\ref{lem_supp}(a). Thus 
\eqref{eq:Lipschitzproperty} is the desired Lipschitz property of the interpolation curve.

Next we aim to prove  \eqref{eq:continuous}. The fact that $\{\rho_t\}_{t\in[0,1]}$ is a Lipschitz curve in 2-Wasserstein distance implies that $\rho_t$ is weakly continuous for $t\in[0,1]$ \cite[Theorem 5.10]{Santambrogio}. Using that $\{\rho_t\}_{t\in[0,1]}$ is uniformly bounded in $L^\infty$ and radially decreasing, we also have that $\rho_t\in L^\infty_t([0,1];BV_x)$, where $BV$ denotes functions of bounded variations. Combining the compactness of $BV$ in $L^1$ and weak continuity of $\rho_t$, we have that $\rho_t\in C_t([0,1]; L^1(\mathbb{R}^n))$. Finally, using that $\rho_t\in L^\infty_{t,x}$,  an interpolation argument shows that
\begin{equation}\label{eq:Lpcontinuity}
    \rho_t\in C_t([0,1];L^p(\R^n))\qquad\mbox{for any $1\le p<\infty$},
\end{equation}
and in particular taking $p=m$ gives that $\mathcal{S}[\rho_t] \in C([0,1])$. Also, since the assumption (W2) implies that $W\in L^q_{loc}(\mathbb{R}^n)$ for some $q>1$, this fact and \eqref{eq:Lpcontinuity} yield that $\mathcal{I}[\rho_t] \in C([0,1])$, and putting the two parts together gives \eqref{eq:continuous}.
\end{proof}

\subsection{Further Regularity}\label{sec:further_reg}

In this subsection, our goal is to show: 
\[
\lim_{t\to 0^+} \frac{\mathcal{E}[\rho_0^s] - \mathcal{E}[\rho_t]}{t} = 0 \quad\text{ and }\quad \lim_{t\to 1^-} \frac{\mathcal{E}[\rho_1^s] - \mathcal{E}[\rho_t]}{1-t} = 0,
\]
without using the gradient flow structure of \eqref{eq:evolution}. To do so, we need to establish some further regularity properties of the interpolation curve in Lemma~\ref{lem:furtherregularity}. Let us start with a simple lemma, saying that although $h_t'(s)\to +\infty$ as $s\to 1^-$ (as given in Lemma~\ref{lemma_h}(d)), the singularity power is indeed the same for all $t\in[0,1]$, as long as $\rho_0$ and $\rho_1$ are both non-degenerate near the origin.

\begin{lemma}\label{lem:nondegeneracy}
Let $\rho_0, \rho_1 \in \mathcal{P}(\R^n) \cap L^\infty(\mathbb{R}^n)$ be strictly radially decreasing and $C^2$ in the interior of their support, and assume they both satisfy the non-degeneracy condition $\Delta \rho_0(0)<0$ and  $\Delta \rho_1(0)<0$. Then there exist $c,C\in (0,\infty)$ and $\bar s\in(0,1)$ that only depend on $\rho_0, \rho_1$, such that
\begin{equation}\label{h'_ineq}
   c(1-s)^{-\frac{n}{n+2}} \le h_t'(s)\le C (1-s)^{-\frac{n}{n+2}} \quad\text{ for all }t\in[0,1] \text{ and }s\in(\bar s, 1).
\end{equation}
\end{lemma}
\begin{proof} It suffices to prove \eqref{h'_ineq} for $t=0,1$, since the general result for $t\in[0,1]$ directly follows from the interpolation $h_t'(s) = (1-t)h_0'(s)+th_1'(s)$. From now on we focus on $\rho_0$, and denote $\rho_0(r):[0,R_0]\to\mathbb{R}$ as the function $\rho_0$ in the radial variable $r$. Since $\rho_0$ is strictly decreasing in $r$ for $0<r<R_0$, \eqref{h_rho} leads to 
\[\rho_0((c_n h_0'(s))^{-1/n}) = h_0(s)\quad\text{ for all }s\in(0,1),
\]
and taking its derivative gives
\begin{equation}\label{hstemp}
 \rho_0'((c_n h_0'(s))^{-1/n})  h_0''(s) = -n c_n^{1/n} h_0'(s)^{2+\frac{1}{n}}  \quad\text{ for all }s\in(0,1).
\end{equation}
By the non-degeneracy condition $\Delta\rho_0(0)<0$, we have that $ \rho_0'(r)<0$ for $0<r\ll 1$, thus $ \rho_0'((c_n h_0'(s))^{-1/n})<0$ for $s\in(0,1)$ sufficiently close to 1. This gives
\begin{equation}\label{h''}
h_0''(s)=-nc_n^{1/n}\frac{h_0'(s)^{2+\frac{1}{n}}}{ \ds \rho_0'\big(\big(c_n h_0'(s))^{-1/n}\big)}> 0 \quad\text{ for $s\in(0,1)$ sufficiently close to 1.}
\end{equation}

Since $\rho_0$ is $C^2$ around the origin with $\nabla\rho_0(0)=0$, there exists some finite $C_0>0$ such that  $\rho_0'(r)\ge -C_0 r$ for $0<r\ll 1$. On the other hand, by the non-degeneracy condition  $\Delta\rho_0(0)<0$, there exists some $c_0>0$ such that $\rho_0'(r)\leq -c_0r$ for $0<r\ll 1$. Plugging these into \eqref{h''}, we know there exists some $s_0\in(0,1)$ such that
$$
\frac{n c_n^{1/n}}{C_0} \leq  \big(h_0'(s)\big)^{-2-\frac{2}{n}} h_0''(s)\le \frac{n c_n^{1/n}}{c_0} \quad\text{ for all }s\in(s_0,1).
$$
Integrating this differential inequality gives \eqref{h'_ineq} for $t=0$ for all $s\in(s_0,1)$. An identical argument can treat the $t=1$ case for all $s\in(s_1,1)$ (where $s_1\in(0,1)$), thus by the interpolation $h_t'(s) = (1-t)h_0'(s)+th_1'(s)$ we obtain \eqref{h'_ineq} for all $t\in[0,1]$ and $s\in(\bar s, 1)$, where $\bar s := \max\{s_0,s_1\} \in (0,1)$. 
\end{proof}

The next lemma deals with regularity of the time derivative $\partial_t \rho_t$.

\begin{lemma}\label{lem:furtherregularity}
Let $\rho_0$ and $\rho_1\in \mathcal{P}(\R^n) \cap L^\infty(\mathbb{R}^n)$ satisfy the assumptions in Lemma~\ref{lem:nondegeneracy}. Let the interpolation curve $\{\rho_t\}_{t\in[0,1]}$ be defined by \eqref{h_t} and \eqref{rho_t}, and consider the vector field $V(x,t):\mathbb{R}^n\times(0,1)\to\mathbb{R}^n$ given in \eqref{continuity} and \eqref{def_v}. Then we have
\begin{equation}\label{reg1}
\partial_t\rho_t=-\nabla\cdot(V\rho)\in C_t([0,1];\mathcal{M}),
\end{equation}
where $\mathcal{M}$ is the space of signed Radon measures.
Here we say $f:\mathbb{R}^n\times[0,1]$ is in $C_t([0,1];\mathcal{M})$ if for any $\phi\in C(\mathbb{R}^d)$, the integral $ \int_{\mathbb{R}^n} f(x,t) \phi(x) dx$ is continuous in $[0,1]$.

\end{lemma}

\begin{proof}
Let us first prove that
\begin{equation}\label{eq_bd1}
 \nabla \cdot(V\rho)\in L^\infty_t([0,1]; L^1(\mathbb{R}^n)).
\end{equation}
For any $x\in \supp\rho_t$ (let us denote $r=|x|$), since $V(x,t) \rho_t(x) = v(r,t)\rho_t(r) \frac{x}{|x|}$,  \eqref{def_v} gives
\[
\begin{split}
\nabla \cdot( V(\cdot,t)\rho_t)&= \partial_r \big(v(r,t)\rho_t(r)\big) + \frac{n-1}{r} v(r,t)\rho_t(r)\\
&= (h_0-h_1)(s_{r,t})+  \frac{r}{n}(h_0-h_1)'(s_{r,t})\partial_r s_{r,t} =: T_1(r)+T_2(r),
\end{split}
\]
where $s_{r,t}$ is defined in \eqref{def_srt}. Since $\|h_i\|_\infty =\|\rho_i\|_\infty <\infty$ for $i=0,1$, we clearly have $|T_1|\leq \max\{\|\rho_0\|_\infty,\|\rho_1\|_\infty\}$. For $T_2$, differentiating \eqref{def_srt2} gives $\partial_r s_{r,t}=\frac{\partial_r\rho_t(r)}{h'_t(s_{r,t})},$ and plugging it into $T_2$ gives
\[
|T_2(r)| \leq \frac{r}{n}\partial_r \rho_t(r) \left|\frac{(h_1-h_0)'(s_{r,t})}{h_t'(s_{r,t})}\right|.\]
Note that for all $t\in[0,1]$ and $r\in(0,R_t)$, using  Lemma~\ref{lem:nondegeneracy} we have
\begin{equation}\label{bd_frac}
\left|\frac{(h_0-h_1)'(s_{r,t})}{h_t'(s_{r,t})}\right| \leq \sup_{s\in(0,1)}\frac{\max\{h_0'(s),h_1'(s)\}}{\min\{h_0'(s),h_1'(s)\}} \leq \max\left\{\frac{h_0'(\bar s)}{h_1'(0)}, \frac{h_1'(\bar s)}{h_0'(0)}, \frac{C}{c}\right\} = C(\rho_0,\rho_1),
\end{equation}
where in the second inequality we used Lemma~\ref{lem:nondegeneracy} as well as the monotonicity of $h_0'$ and $h_1'$ by Lemma~\ref{lemma_h}(a). 
%
%
Hence for any $t\in[0,1]$, 
\begin{equation}\label{eq:bound}
\int_{\mathbb{R}^n}|\nabla \cdot (V\rho_t)|\,dx\le \int_0^{R_t} |T_1(r)+T_2(r)| \omega_n r^{n-1}dr \leq C(\rho_0,\rho_1),
\end{equation}
where in the last inequality we used the uniform bound of $R_t$ in Lemma~\ref{lem_supp}, and the fact that 
$\int_0^{R_t} |\partial_r \rho_t(r)| dr \leq \|\rho_t\|_\infty \leq \max\{\|\rho_0\|_{\infty},\|\rho_1\|_{\infty}\}$. This finishes the proof of \eqref{eq_bd1}.

Next we aim to show
\begin{equation}\label{eq_bd2}
\partial_t \nabla \cdot(V\rho)\in L^\infty_t([0,1]; W^{-1,\infty}(\mathbb{R}^n)).
\end{equation}
 Taking a radial test function $\psi\in C^1(\R^n)$, we have
\[
\int_{\R^n}\psi \nabla \cdot(V(x,t)\rho_t)\,dx=-\int_{\R^n}\nabla \psi \cdot V(x,t)\rho_t\,dx=\int_0^\infty  \psi'(r)(h_1-h_0)(s_{r,t}) c_n r^n \,dr,
\]
where we use the first identity of \eqref{def_v} in the second equality.
Differentiating in $t$, we have
\begin{equation}\label{dt_psi}
\frac{d}{dt}\int_{\R^n}\psi \nabla \cdot(V(x,t)\rho_t)\,dx=\int_0^\infty  \psi'(r)(h_1-h_0)'(s_{r,t}) \partial_t  s_{r,t} \, c_n r^n \,dr.
\end{equation}
Note that $\partial_t  s_{r,t}$ can be explicitly computed as follows.  
Differentiating \eqref{def_srt} in $t$ (and note that its right hand side is independent of $t$) gives 
$$
0 = \partial_t (h'_t(s_{r,t}))= (h'_1-h'_0)(s_{r,t}) + h''_t(s_{r,t})\partial_t s_{r,t},
$$
thus
\begin{equation}\label{dt_srt}
\partial_t s_{r,t}=-\frac{(h_1'-h_0')(s_{r,t})}{h''_t(s_{r,t})}=n c_n^{1/n} \partial_r \rho_t(r)\frac{(h_1'-h_0')(s_{r,t})}{(h'_t(s_{r,t}))^{2+1/n}},
\end{equation}
where we use that $\frac{1}{h_t''(s_{r,t})} = - \frac{\partial_r \rho_t(r)}{nc_n^{1/n} h_t'(s_{r,t})^{2+1/n} }$ in the last inequality, which follows from  \eqref{hstemp} (note that even though the equation is stated for $\rho_0$, it indeed works for $\rho_t$ as well, which is known to be strictly decreasing in its support) and \eqref{def_v}.

Plugging \eqref{dt_srt} into \eqref{dt_psi}, the left hand side of \eqref{dt_psi} can be bounded as
\begin{equation}\label{eq_psi_temp}
\begin{split}
\left|\frac{d}{dt}\int_{\R^n}\psi \nabla \cdot(V(x,t)\rho_t)\,dx\right| &\leq \int_0^\infty  \psi'(r) |\partial_r \rho_t| \left(\frac{(h_1'-h_0')(s_{r,t})}{h_t'(s_{r,t})}\right)^2  h_t'(s_{r,t})^{-1/n} c_n^{1+\frac{1}{n}} r^n \, dr \\& \leq C(\rho_0,\rho_1) \, \|\psi\|_{C^1},
\end{split}
\end{equation}
where in the second inequality we use \eqref{bd_frac} to control the fraction, and also used $\int_0^\infty |\partial_r \rho_t| dr = \int_0^{R_t} |\partial_r \rho_t| dr \leq \max\{\|\rho_0\|_\infty,\|\rho_1\|_\infty\}$, as well as the fact that $h_t'(s_{r,t})\geq \min\{h_0'(0),h_1'(0)\}>0$. 
Since the right hand side of \eqref{eq_psi_temp} is independent of $t$, this concludes the proof of \eqref{eq_bd2}.

Finally, we put \eqref{eq_bd1} and \eqref{eq_bd2} together, and apply the Aubin-Lions type Lemma. The compactness of $L^1(\mathbb{R}^d)$ in $\mathcal{M}$ implies
$
\nabla\cdot(V\rho)\in C_t([0,1];\mathcal{M}),
$
for a similar proof see \cite[Lemma 8.1.2]{AGS}.
\end{proof}

Using the previous regularity lemma, for the interpolation curve between any two radial stationary solutions $\rho_0$ and $\rho_1$, we show its energy functional $\mathcal{E}[\rho_t]$ has a zero right derivative at $t=0$.

\begin{proposition}\label{prop:general}
Let $\rho_0$ and $\rho_1$ be two radially symmetric steady states in the sense of Definition~\ref{def:gendef}. We consider $\rho_t:[0,1]\to \mathcal{P}(\R^n)$ the interpolation given by \eqref{rho_t}. Then for $m\geq 2$, we have
    \begin{equation}\label{goal_e}
        \lim_{t\to 0^+} \frac{\mathcal{E}[\rho_t]-\mathcal{E}[\rho_0]}{t}=0.
    \end{equation}
\end{proposition}
\begin{proof}
We decompose the energy into the entropy and interaction part. For the entropy $\mathcal{S}$, we use \eqref{Internal Energy} with $\Phi(s) = \frac{1}{m-1}s^m$ to rewrite it as
$$
\mathcal{S}[\rho_t] = \frac{1}{m-1}\int_{\R^n} \rho_{t}^m(x)\,dx=\frac{m}{m-1}\int_0^1\big( (1-t)h_0(s)+t h_1(s)\big)^{m-1}\,ds \quad\text{ for all }t\in(0,1),
$$
thus the finite difference can be written as
\[
\frac{\mathcal{S}[\rho_t] - \mathcal{S}[\rho_0]}{t} = \frac{m}{m-1} \int_0^1 \frac{\big( (h_0(s)+t(h_1-h_0)(s)\big)^{m-1} - h_0(s)^{m-1}}{t} ds.
\]
Note that for all $s\in(0,1)$, the integrand converges to $(m-1) h_0^{m-2}(h_1-h_0)$ as $t\to 0^+$. In addition, since $m\geq 2$, due to the convexity of $h\mapsto h^{m-1}$, the absolute value of the integrand is bounded by $(m-1)\max\{\|h_0\|_\infty, \|h_1\|_\infty\}^{m-2}|h_1-h_0|$ for all $t\in(0,1)$, which is finite since $\|h_i\|_\infty = \|\rho_i\|_\infty$ for $i=0,1$. Thus Lebesgue's dominated convergence theorem gives
$$
\lim_{t\to 0+}
\frac{\mathcal{S}[\rho_t] - \mathcal{S}[\rho_0]}{t} =m\int_0^1 h_0(s)^{m-2} (h_1(s)-h_0(s))\,ds.
$$

Next we deal with the interaction energy $\mathcal{I}[\rho_t]=\frac{1}{2}\int_{\mathbb{R}^n} \rho_t (\rho_t * W)dx$, and aim to show that
\begin{equation}\label{goal_i}
\lim_{t\to 0+}
\frac{\mathcal{I}[\rho_t] - \mathcal{I}[\rho_0]}{t} =-m\int_0^1 h_0(s)^{m-2} (h_1(s)-h_0(s))\,ds.
\end{equation}
Once this is done, adding the two inequalities above directly yields \eqref{goal_e}, finishing the proof.

We use the representation of the interpolation curve by the continuity equation. Notice that for any $t\in(0,1)$, we have
\begin{equation}\label{eq_temp_i}
\frac{\mathcal{I}[\rho_t] - \mathcal{I}[\rho_0]}{t}=\frac{1}{t}\int_0^t \int_{\R^n}(\partial_t \rho_t) (W*\rho_t) \,dxdt = \frac{1}{t}\int_0^t \int_{\R^n}-\nabla\cdot(\rho_t V_t)  (W*\rho_t) \,dxdt.
\end{equation}
Let us point out that \begin{equation}\label{conti_w}
W*\rho_t\in C(\mathbb{R}^n\times[0,1]).
\end{equation}
To see this, recall that in \eqref{eq:Lpcontinuity} we showed that $\rho_t\in C_t(L^p_x)$ for any $p\in[1,\infty)$. Combining this with the property that $W\in L^q_{loc}$ for some $q>1$ (by (W2)), as well as the fact that $\{\rho_t\}_{t\in[0,1]}$ are uniformly compactly supported by Lemma~\ref{lem_supp}(a),  we have \eqref{conti_w}.

By Lemma~\ref{lem:furtherregularity}, we have $\partial_t\rho_t= - \nabla\cdot(\rho_t V_t)\in C_t([0,1];\mathcal{M})$, thus
$$
\int_{\R^n} -\nabla\cdot(\rho_t V_t)(W*\rho_t ) \,dx\;\;\mbox{is continuous in $t$ for $t\in[0,1]$.}
$$
Using the continuity property in $t$, we can send $t\to 0^+$ in \eqref{eq_temp_i} to obtain
\begin{equation}\label{eq_temp_i2}
\begin{split}
\lim_{t\to0^+}\frac{\mathcal{I}[\rho_t] - \mathcal{I}[\rho_0]}{t}&= \int_{\R^n}-\nabla\cdot(\rho_0 V_0) (W*\rho_0) \,dx\\
& =\int_{\supp\rho_0} \rho_0 V_0 \cdot \nabla (W*\rho_0) \,dx\\
& =\int_{\supp\rho_0} \rho_0 V_0 \cdot \left(-\frac{m}{m-1} \nabla \rho^{m-1}\right) \,dx\\
&=  -m\int_{\supp\rho_0} \rho_0^{m-1} \nabla \rho_0 \cdot V_0  \,dx,
\end{split}
\end{equation}
where we used Definition~\ref{def:gendef} in the second-to-last inequality.

Writing $\rho_0(x)=\rho_0(r)$, the above integral can be written in radial coordinates as
\begin{equation}\label{identity0}
m\int_{\supp\rho_0}\rho_0^{m-1}\nabla  \rho_0 \cdot V_0\,dx=m\int_0^{R_0} \rho_0(r)^{m-1} \rho_0'(r)v(r,0)  \omega_n r^{n-1}\,dr,
\end{equation}
where $v(r,0)$ is given by \eqref{def_v}.
Next, we want to take the monotone change of variables 
$
    c_n r^{n}=\frac{1}{h_0'(s)}
$
to express the right hand side of \eqref{identity0} as an integral of $s$. With this change of variables, we have the identities
$$
\rho_0(r)=h_0(s),\qquad v(r,0)=\big(c_n h_0'(s)\big)^{-1/n}\frac{h_0(s)-h_1(s)}{nh_0(s)}
$$
and
$$
\omega_n r^{n-1}dr=-\frac{h_0''(s)}{h_0'(s)^2}ds=\frac{n c_n^{1/n}  h_0'(s)^{1/n}}{\rho_0'((c_n h_0'(s))^{-1/n})}ds=\frac{n c_n^{1/n}  h_0'(s)^{1/n}}{\rho_0'(r)}ds,
$$
where we used the expression for $h''(s)$ in \eqref{h''}. Plugging the above into \eqref{identity0}, and note that we have $s\to 1$ as $r\to0$, and $s\to0$ as $r\to R_0$. This gives
$$
m\int_0^{R_0} \rho_0(r)^{m-1}\rho_0'(r)v(r,0) \omega_n r^{n-1}\,dr=m\int_0^1 h_0(s)^{m-2}(h_1(s)-h_0(s))\,ds.
$$
Combining this with \eqref{eq_temp_i2} and \eqref{identity0} gives \eqref{goal_i}, finishing the proof.
\end{proof}

\subsection{Convexity of the interaction energy along the curve}\label{sec_convexity}
In this section, we aim to prove the following proposition:
\begin{proposition}\label{prop_convex_nd} 
Let $\rho_0, \rho_1\in \mathcal{P}(\R^n) \cap C(\mathbb{R}^n)$ be two radially decreasing probability densities on $\mathbb{R}^n$ that are not identical.  Consider the interpolation curve $\{\rho_t\}_{t\in[0,1]}$ as given in \eqref{h_t} and \eqref{rho_t}. Then the function $t\mapsto \mathcal{I}[\rho_t]$ is strictly convex for $t\in (0,1)$. 
\end{proposition}

\begin{proof}

Most of this proof is devoted to the convexity of $t\mapsto \mathcal{I}[\rho_t]$, except that at the very end we will improve the convexity into strict convexity. By an identical argument as in the first paragraph of the proof of Proposition~\ref{prop_convex_1d}, it suffices to obtain the convexity of $t\mapsto \mathcal{I}[\rho_t]$ for each interaction potential of the form $W=W_a$ for all $a>0$, with $W_a(|x|)$ given by the step function \eqref{def_wa}.

For the potential $W=W_a$, using \eqref{h_rho}, we rewrite the interaction energy in $\mathbb{R}^n$ as
\begin{equation*}
\begin{split}
\mathcal{I}[\rho_t] &=\frac{1}{2} \int_0^1 \int_0^1 h_t'(s_1) h_t'(s_2) \left|\Big\{(x,y): |x|\leq  (c_n h_t'(s_1))^{-1/n}, |y|\leq (c_n h_t'(s_2))^{-1/n}, |x-y|>a\Big\}\right| ds_1 ds_2.
\end{split}
\end{equation*}
Denote the integrand by $I(t; s_1, s_2)$. In order to show that $\mathcal{I}[\rho_t]$ is convex in $t$, it suffices to show that $I(t; s_1, s_2)$ is a convex function of $t$ for a.e. $s_1, s_2 \in (0,1)$. Let us rewrite $I$ as
\begin{equation*}
\begin{split}
I(t; s_1, s_2) &= h_t'(s_1) h_t'(s_2) \left|\Big\{(x,y): \left|\frac{x}{a}\right|\leq  (c_n a^{n} h_t'(s_1))^{-1/n}, \left|\frac{y}{a}\right|\leq (c_n a^{n}  h_t'(s_2))^{-1/n}, \left|\frac{x}{a}-\frac{y}{a}\right|>1\Big\}\right|\\[0.1cm]
&= a^{2n} h_t'(s_1) h_t'(s_2) \left|\Big\{(x,y): |x|\leq  (c_n a^{n} h_t'(s_1))^{-1/n}, |y| \leq (c_n a^{n}  h_t'(s_2))^{-1/n}, |x-y|>1\Big\}\right|.
 \end{split}
\end{equation*}
For any fixed $s_1, s_2\in(0,1)$, let us introduce 
\[
R(t) := (c_n a^n h_t'(s_1))^{-1/n}~\text{ and }~r(t) :=(c_n a^n h_t'(s_2))^{-1/n} \quad\text{ for }t\in[0,1],
\] so that we can rewrite $I$ in terms of $R(t)$ and $r(t)$:
\begin{equation}\label{temp_i}
I(t; s_1, s_2) = c_n^{-2} R(t)^{-n} r(t)^{-n} \left|\Big\{(x,y): |x|\leq R(t), |y| \leq r(t), |x-y|>1\Big\}\right| =: I(R(t),r(t)).
\end{equation}
 From now on, by a slight abuse of notation, we will denote the function as $I(R(t), r(t))$. Recall that $R(t)^{-n}$ and $r(t)^{-n}$ are both affine functions of $t$, since $h_t'(s_1)$ and $h_t'(s_2)$ are both affine in $t$.  
 
 For almost every $(s_1, s_2)\in(0,1)^2$, we are in one of the following cases. (By an identical argument as in Proposition~\ref{prop_convex_1d}, we can show that $|R(t)-r(t)|=1$ or $R(t)+r(t)=1$ only happen for a zero measure set of $(s_1,s_2)\in(0,1)^2$.)
 
Case 1. $R(t)+r(t)<1$. In this case we have $I(R(t), r(t))=0$, and it remains zero in a small interval containing $t$, thus $\frac{d^2}{dt^2} I(R(t), r(t)) = 0$. 

Case 2. $|R(t)-r(t)| > 1$. Without loss of generality, assume that $R(t) - r(t)>1$. Then we have
\begin{equation*}
\begin{split}
I(t; s_1, s_2) &= 1 - c_n^{-2} R(t)^{-n} r(t)^{-n} \left|\Big\{(x,y): |x|\leq R(t), |y| \leq r(t), |x-y|<1\Big\}\right|\\
&= 1 - c_n^{-2} R(t)^{-n} r(t)^{-n} \int_{B(0,r(t))}\int_{B(y,1)} dx dy \text{\quad(since $B(y,1) \subset B(0,R)$ for $y\in B(0,r)$)}\\
&= 1 -  R(t)^{-n}.
\end{split}
\end{equation*}
And since $R(t)^{-n}$ is an affine function of $t$, this leads to $\frac{d^2}{dt^2} I(R(t), r(t)) = 0$ in some interval containing $t$.

Case 3. $|R(t)-r(t)| < 1$ and $R(t)+r(t)>1$. The analysis in this case will be much more involved compared to the 1D proof, since we no longer have an explicit formula of $I$ as a function of $t$. 

Denote $\alpha:=\frac{d}{dt}(R^{-n})$ and $\beta:=\frac{d}{dt}(r^{-n})$. Note that $\alpha, \beta$ are constants, which could be positive or negative. We will express $\frac{d^2}{dt^2}I$ as a quadratic function of $\alpha$ and $\beta$ (where the coefficients depends on $R, r$ and $I_{RR}, I_{Rr}, I_{rr}$), and investigate the coefficients of this quadratic function. 

Let us start with the first derivative $\frac{d}{dt} I(R(t), r(t))$:
  \begin{equation}\label{I_first_derivative_t}
    \frac{d}{dt}I(R(t),r(t))=I_R(R(t),r(t)) \, R'(t)+I_r(R(t), r(t))\, r'(t).
  \end{equation}
  By definition of $\alpha$ and $\beta$, we have
  \[
    \alpha=-nR^{-n-1}R'(t), \quad \beta=-nr^{-n-1}r'(t),
  \]
  thus 
  \[
    R'(t)=-\frac{\alpha}{n}R(t)^{n+1}, \quad r'(t)=-\frac{\beta}{n}r(t)^{n+1}.
  \]
  Plugging these into \eqref{I_first_derivative_t} gives the following, where we compress the dependence on the variables for notational simplicity:
  \[
  \frac{d}{dt}I =-\frac{1}{n}(\alpha R^{n+1}I_R+\beta r^{n+1}I_r).
  \]
 Taking another derivative in $t$ on both sides gives 
  \begin{equation}\label{eq_2nd_der}
    \begin{split}
     \frac{d^2}{dt^2}I & =-\frac{1}{n}\Big(\alpha(R^{n+1}I_R)_R R'+\alpha(R^{n+1}I_R)_rr'+\beta(r^{n+1}I_r)_RR'+\beta(r^{n+1}I_r)_rr'\Big)\\
                               & =\frac{1}{n^2}\Big(\alpha^2 \underbrace{R^{n+1}(R^{n+1}I_R)_R}_{=: u} +2\alpha\beta \underbrace{R^{n+1}r^{n+1}I_{Rr}}_{:= v} +\beta^2 \underbrace{r^{n+1}(r^{n+1}I_r)_r}_{=: w}\Big).
     \end{split}
  \end{equation}
  With $u, v, w$ defined as above (all are functions of $R,r$), $\frac{d^2}{dt^2}I$ can be written as a quadratic function of $\alpha, \beta$ as 
  \[n^2 \frac{d^2}{dt^2}I(R(t),r(t)) =  u \alpha^2 + 2v \alpha\beta + w \beta^2.\]
   In order to show that it is nonnegative for all $\alpha, \beta\in\mathbb{R}$, it suffices to show that $u, w> 0$ and $uw - vˆ2 > 0$. By Lemma~\ref{lemI_1}, which we will prove right after this proof, we have the explicit expressions of $u,v,w$:
   \[
   \begin{split}
   u &= R^{n+1}(R^{n+1}I_R)_R = \frac{R^{n+1}}{r^{n+1}} (R^{n+1} r^{n+1}I_R)_R = \frac{R^{n+1}}{r^{n+1}} \tilde c_n Rr \,S(R,r)^{\frac{n-1}{2}}> 0,\\[0.1cm]
   v &= R^{n+1}r^{n+1}I_{Rr} = -\frac{\tilde c_n}{2} (R^2 + r^2 - 1) S(R,r)^{\frac{n-1}{2}},\\[0.1cm]
   w &= r^{n+1}(r^{n+1}I_r)_r =\frac{r^{n+1}}{R^{n+1}}  (R^{n+1} r^{n+1}I_r)_r =\frac{r^{n+1}}{R^{n+1}} \tilde c_n Rr \,S(R,r)^{\frac{n-1}{2}}> 0,
   \end{split}
   \]
   where $\tilde c_n$ and $S(R,r)$ are defined in Lemma~\ref{lemI_1}.
  A direct computation gives 
  \[
  uw - vˆ2 = \tilde c_n^2 S(R,r)^{n-1} \left( R^2 r^2 -\frac{1}{4}(R^2 + r^2 - 1)^2  \right) = \frac{\tilde c_n^2}{4} S(R,r)^n>0, 
  \]
  where we used the fact that $(R^2 + r^2 - 1)^2 - 4R^2 r^2 = -S(R,r)$ in the last identity. Since $u,w>0$ and $uw - v^2 > 0$, we then have $\frac{d^2}{dt^2} I\geq 0$ for all $\alpha, \beta \in \mathbb{R}$, finishing the convexity proof.

  Finally, it remains to upgrade the convexity into strict convexity when $\rho_0$ and $\rho_1$ are not identical. This can be done in the same way as the end of the proof of Proposition~\ref{prop_convex_1d}:  If $\rho_0$ and $\rho_1$ are not identical, without loss of generality we can assume $h_0'(s) < h_1'(s)$ in some small open interval containing $s_0\in(0,1)$. Then for all $a>0$ that is sufficiently small and $s_1, s_2$ sufficiently close to $s_0$, we have that $R(t) := (c_n a^n h_t'(s_1))^{-1/n}$ and $r(t) :=(c_n a^n h_t'(s_2))^{-1/n}$ belong to Case 3. In addition, for these $s_1, s_2$ we have $\partial_t h_t'(s) = h_1'(s)-h_0'(s) > 0$ for $t\in (0,1)$, implying that $R'(t)<0$ and $r'(t)<0$, thus $\alpha,\beta$ in \eqref{eq_2nd_der} are both strictly positive. Combining this with the fact that $uw - v^2 > 0$, we have that $\frac{\partial^2}{\partial t^2} I(t; s_1, s_2)>0$ for a positive measure of $(s_1, s_2)$ for all sufficiently small $a>0$, implying the strict convexity of $t\mapsto \mathcal{I}[\rho_t]$.
\end{proof}

Finally, it remains to prove Lemma~\ref{lemI_1}. Recall that in the proof above, for any $R,r>0$ satisfying $|R-r|< 1< R+r$, the function $I(R,r)$ is given by \eqref{temp_i}, namely
\begin{equation}\label{def_I_int1}
\begin{split}
I(R,r) &= c_n^{-2} R^{-n} r^{-n} \left|\Big\{(x,y): |x|\leq R, |y| \leq r, |x-y|>1\Big\}\right| \\
&= c_n^{-2} R^{-n} r^{-n} \int_{B(0,1)^c} \left(1_{B(0,R)} * 1_{B(0,r)}\right)(x) \,dx.
\end{split} 
\end{equation}
Next we will obtain the following explicit expressions on the second derivatives of $I$.
\begin{lemma}\label{lemI_1}
   Let $r,R>0$, and assume they satisfy $|R-r|< 1< R+r$.  For the function $I(R,r)$ given by \eqref{def_I_int1}, let 
\[   \begin{split}
   &U:=(R^{n+1}r^{n+1}I_{R})_R,\\[0.1cm]
   &V:=R^{n+1}r^{n+1}I_{Rr},\\[0.1cm]
   &W:=(R^{n+1}r^{n+1}I_{r})_r.
   \end{split}
\]Then we have the identities
   \begin{align}
     & U=W=\tilde c_n Rr \,S(R,r)^{\frac{n-1}{2}}> 0,\label{eq:uw}\\
     & V= -\frac{\tilde c_n}{2} (R^2 + r^2 - 1) S(R,r)^{\frac{n-1}{2}}, \label{eq:v}
      \end{align}
  
   where $\tilde c_n := \frac{c_{n-1}\omega_n}{2^{n-1}c_n^2} = \frac{c_{n-1}}{2^{n-1}n c_n}$, and $S(R,r)$ is given by
 \begin{equation}\label{def_S}
S(R,r):=(R+r+1)(-R+r+1)(R-r+1)(R+r-1).
\end{equation}
\end{lemma}

\noindent\textbf{Remark.} Note that $S(R,r)>0$ for $|R-r|< 1< R+r$, and it has a geometric meaning: by Heron's formula, the area of the triangle with side lengths $R, r, 1$ is  $\frac{1}{4}S(R,r)^{1/2}$.

\begin{proof}
\noindent\textbf{Step 1. Equivalent expressions of $I(R,r)$.} Before we take derivatives on $I$, let us begin by rewriting $I(R,r)$ into some equivalent expressions. Since the integrand in \eqref{def_I_int1} is radially symmetric, one can write the integral in radial variable $s$ as
\[
I(R,r) =c_n^{-2} R^{-n} r^{-n}   \int_1^{r+R} A(R,r; s) \omega_n s^{n-1} ds,
\]
where we denote by $A(R,r;s)$  the volume of the intersection of two balls (in $\mathbb{R}^n$) of radii $R$ and $r$, with their centers separated by distance $s$. (The integral has upper limit $r+R$ since $A(R,r; s)\equiv 0$ when $s\geq R+r$).

Although this is the most straightforward way to express $I$,  its second derivative would involve second derivatives of $A$, whose analytical expression is difficult to obtain.  To circumvent this difficulty, we express $I(R,r)$ in another way as follows. By introducing  $f:=1_{B(0,R)}$, $g:=1_{B(0,r)}$ and $h:=1_{B(0,1)}$, note that \eqref{def_I_int1} can be rewritten as
\begin{equation*}
\begin{split}
I(R,r) &=  c_n^{-2} R^{-n} r^{-n} \int_{\mathbb{R}^n} (f*g)(1-h)dx\\
& =c_n^{-2} R^{-n} r^{-n} \left( \|f*g\|_1 - \int_{\mathbb{R}^n} (g*h) f dx\right)\quad \text{(since $\int_{\mathbb{R}^n} (f*g)h dx= \int_{\mathbb{R}^n} (g*h)f dx$)}\\
&= c_n^{-2} R^{-n} r^{-n} \left( \|f\|_1 \|g\|_1  - \|g\|_1 \|h\|_1 + \int_{\mathbb{R}^n} (g*h) (1-f) dx\right)\\
&=  1 - R^{-n} + c_n^{-2} R^{-n} r^{-n} \int_{B(0,R)^c} 1_{B(0,r)} * 1_{B(0,1)} dx,
\end{split}
\end{equation*}
and writing the last integral in radial variable gives
\begin{equation} \label{I_def_R}
I(R,r)=  1 - R^{-n} + c_n^{-2} R^{-n} r^{-n} \int_R^{r+1} A(r,1; s) \omega_n s^{n-1} ds,
\end{equation}
again the integral has upper limit $r+1$ since $A(r,1;s)\equiv 0$ for all $s\geq r+1$. 
Differentiating, we get
\begin{equation} \label{I_R_derivative}
I_R= n R^{-n-1} - nc_n^{-2} R^{-n-1} r^{-n} \int_R^{r+1} A(r,1; s) \omega_n s^{n-1}ds -c_n^{-2}\omega_n R^{-1} r^{-n} A(r,1; R) ,
\end{equation}
\begin{equation} \label{I_Rr_derivative}
\begin{split}
I_{Rr}&=  -c_n^{-2}\omega_n R^{-1} r^{-n} A_r(r,1; R)+nc_n^{-2}\omega_n R^{-1} r^{-n-1} A(r,1; R)\\
&\qquad +nc_n^{-2} R^{-n-1} r^{-n-1} \int_R^{r+1} \left(nA(r,1; s)-rA_r(r,1; s)\right)\omega_n s^{n-1}ds,
\end{split}
\end{equation}
where we have used that $A(r,1;r+1)=0$. We will use \eqref{I_R_derivative} and \eqref{I_Rr_derivative} to obtain explicit expressions of $U$ and $V$ in \eqref{eq:uw} and \eqref{eq:v}. Once this is done, since $I(R,r)$ is symmetric in $R, r$, a parallel argument gives that $W$ is equal to the right hand side of \eqref{eq:uw} (with $R$ and $r$ switched). Since $S(R,r)$ is also symmetric in $R, r$, we then have that $U=W$.
%

\noindent\textbf{Step 2. Partial derivatives of $A(r,1;s)$.} To prepare for the computations later, in this step we find an integral representation of the term $A(r,1;s)$ (which appears in the integrand of \eqref{I_def_R}), and find its first derivatives with respect to $r$ and $s$. Note that $s\in (R,r+1)$ and the assumption $|R-r| < 1 < R+r$ imply that $|s-r|<1<s+r$. In other words, there exists a triangle with side lengths $r, s, 1$.

Consider two balls with radius $r$ and $1$ centered at $O = (0,\textbf{0})$ and $P = (s, \textbf{0})$ respectively ($\textbf{0}$ denotes the zero vector in $\mathbb{R}^{n-1}$). By definition, $A(r,1;s)$ is the volume of the intersection of the two balls in $\mathbb{R}^n$. To compute $A(r,1;s)$, we take $Q$ to be any point on both spheres.  Let $l(r;s)$ be the distance from $Q$ to the $x_1$-axis, and $s_1(r;s)$ be the $x_1$ coordinate of $Q$. The definitions of $A(r,1;s), l(r;s)$ and $s_1(r;s)$ are illustrated in Figure~\ref{fig_circ}.
\begin{figure}[h!]
\begin{center}
\includegraphics[scale=0.9]{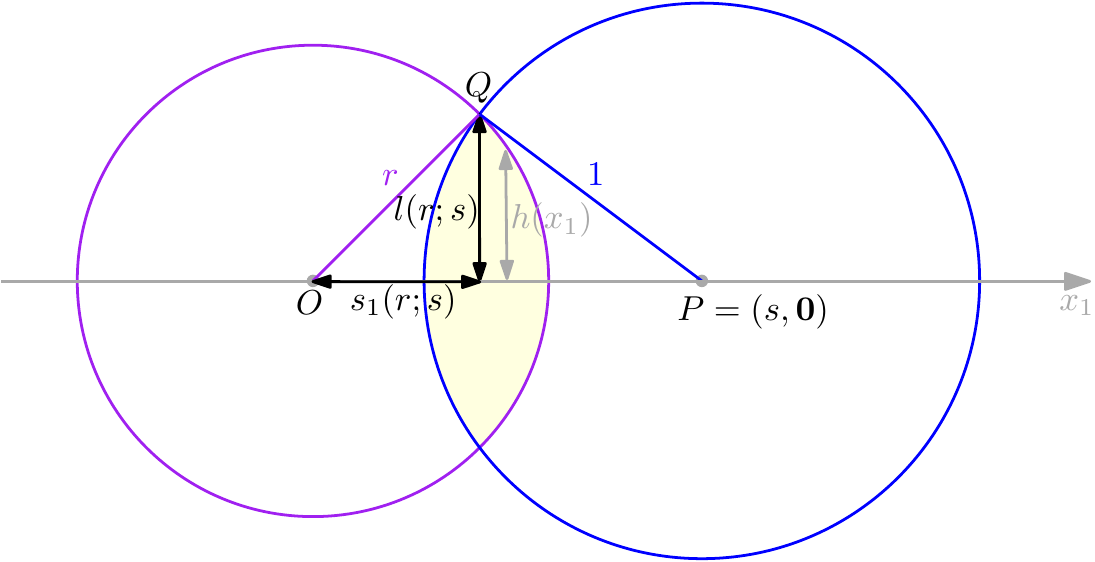}
\caption{\label{fig_circ}Illustration of the definitions of  $A(r,1; s)$, $l(r;s)$ and $s_1(r;s)$. For two balls in $\mathbb{R}^n$ with radii $r$ and $1$ respectively, $A(r,1;s)$ denotes the volume of their intersection in $\mathbb{R}^n$ (colored in yellow). Let $Q$ be any intersection of the two spheres. We define $l(r;s)$ as the distance from $Q$ to the $x_1$-axis, and $s_1(r;s)$ as the $x_1$ coordinate of $Q$.}
\end{center}
\end{figure}

By law of cosines, we have $\cos \angle QOP = \frac{r^2 + s^2-1}{2rs}$. This immediately yields 
    \begin{equation}\label{s1_comp}
    s_1(r;s)=r \cos \angle QOP = \frac{s^2+r^2-1}{2s},
    \end{equation}
    and 
    \[ l(r;s) = \sqrt{r^2-s_1^2(r;s)} = \frac{\sqrt{4r^2s^2-(s^2+r^2-1)^2}}{2s}.
    \]
We note that the area of the triangle $OPQ$ is given by
    \begin{equation}\label{eq:l&S}
        \frac{sl(r;s)}{2}=\frac{1}{4}S(s,r)^{1/2},    
    \end{equation}
where the right hand side is given by Herron's formula \eqref{def_S}.

    For any hyperplane in $\mathbb{R}^n$ perpendicular to the $x_1$-axis, its intersection with the yellow region is non-empty if and only if the $x_1$ coordinate of the hyperplane is between $s-1$ and $r$. In this case, the intersection is an $(n-1)$-dimension ball with radius $h(x_1)$, given by the following (the definition of $h(x_1)$ is also illustrated in Figure~\ref{fig_circ}):
    \[
      H(x_1)=\begin{cases}
           \sqrt{1-(s-x_1)^2} & \text{ for } s-1\le x_1\le s_1(r;s),\\
           \sqrt{r^2-x_1^2}&  \text{ for } s_1(r;s)\le x_1\le r.
       \end{cases}
    \]
    This allows us to express $A(r,1;s)$ as an integral in $x_1$:
    \begin{equation}\label{eqI_5}
      \begin{split}
       A(r,1;s) & =\int_{s-1}^{r}c_{n-1}H^{n-1}(x_1)dx_1\\
                   & =c_{n-1}\left(\int_{s-1}^{s_1(r;s)}(1-(s-x_1)^2)^{\frac{n-1}{2}}dx_1+\int_{s_1(r;s)}^{r}(r^2-x_1^2)^{\frac{n-1}{2}}dx_1\right)\\
                   & =c_{n-1}\left(\int_{s-s_1(r;s)}^{1}(1-y^2)^{\frac{n-1}{2}}dy+\int_{\frac{s_1(r;s)}{r}}^{1} r^n(1-y^2)^{\frac{n-1}{2}}dy\right).
       \end{split}
    \end{equation}
  We then compute the partial derivatives of $A$ as follows. (Note that $s_1(r;s)$ is a function of $r, s$, even though the dependence are compressed below for notational simplicity.)
    \begin{equation}\label{eq:As}
       \begin{split}
          A_s(r,1;s) & =c_{n-1}\left(-(1-(s-s_1)^2)^{\frac{n-1}{2}}\frac{\partial}{\partial s}(s-s_1)-r^n\left(1-\frac{s_1^2}{r^2}\right)^{\frac{n-1}{2}}\frac{\partial}{\partial s}\left(\frac{s_1}{r}\right)\right)\\
                          & =-c_{n-1}l^{n-1}\left(1-\frac{\partial s_1}{\partial s}+\frac{\partial s_1}{\partial s}\right)\\
                           & =-c_{n-1}l(r;s)^{n-1},
       \end{split}
    \end{equation}
    where in the second equality we used that $l(r;s) = \sqrt{1-(s-s_1)^2}$ and $l(r;s) = \sqrt{r^2-s_1^2}$. 
    
    Likewise, we compute $A_r$ as
    \begin{equation}\label{eq:Ar}
       \begin{split}
          A_r(r,1;s) & =c_{n-1}\left(-(1-(s-s_1)^2)^{\frac{n-1}{2}}\frac{\partial}{\partial r}(s-s_1)-r^n\left(1-\frac{s_1^2}{r^2}\right)^{\frac{n-1}{2}}\frac{\partial}{\partial r}\left(\frac{s_1}{r}\right) + \int_{s_1/r}^{1}nr^{n-1}(1-y^2)^{\frac{n-1}{2}}dy\right)\\
                         & =-c_{n-1}l(r;s)^{n-1}\left(-\frac{\partial s_1}{\partial r}+r\frac{1}{r}\frac{\partial s_1}{\partial r}-r\frac{s_1}{r^2}\right)+c_{n-1} \int_{s_1/r}^{1}nr^{n-1}(1-y^2)^{\frac{n-1}{2}}dy\\
                         & =c_{n-1}l(r;s)^{n-1}\frac{s_1(r;s)}{r}+c_{n-1} \int_{\frac{s_1(r;s)}{r}}^{1}nr^{n-1}(1-y^2)^{\frac{n-1}{2}}dy.
       \end{split}
    \end{equation}

\noindent\textbf{Step 3. Explicit formula for $U, W$.} In this step we prove \eqref{eq:uw}. As we discussed at the end of \textbf{Step 1}, it suffices to focus on $U$ and aim to show that
\[
U:=(R^{n+1}r^{n+1}I_{R})_R = \tilde c_n Rr \,S(R,r)^{\frac{n-1}{2}},
\]
and the identity for $W$ would follow from a parallel argument.

Using \eqref{I_R_derivative}, we get that
\begin{equation*}
\begin{split}
U&= nc_n^{-2}  r A(r,1; R) \omega_n R^{n-1} -nc_n^{-2}\omega_n R^{n-1} r A(r,1; R)-c_n^{-2}\omega_n R^{n} r A_R(r,1; R)\\
&=-c_n^{-2}\omega_n R^{n} r A_R(r,1; R).
\end{split}
\end{equation*}
By \eqref{eq:As}, we have $A_R(r,1,R) = -c_{n-1}l(r;R)^{n-1}$. Therefore,
\[
U =  \frac{c_{n-1} \omega_n}{c_n^2} rR^n l(r;R)^{n-1} = \frac{c_{n-1}}{2^{n-1} n c_n} Rr \,S(R,r)^{\frac{n-1}{2}}.
\]
where in the second step we used the fact $\omega_n = nc_n$ as well as the relationship $\frac{1}{2} R l(r;R) = \frac{1}{4}S(R,r)^{1/2}$ in \eqref{eq:l&S}, since both sides give the area of the triangle with side lengths $r,1,R$.

\noindent\textbf{Step 4. Explicit formula for $V$.} In this final step we aim to show \eqref{eq:v}. We compute $V$ as follows:
        \begin{equation}\label{eqI_9}
           \begin{split}
               V=R^{n+1}r^{n+1}I_{Rr}& =-c_n^{-2}\omega_n R^{n} r A_r(r,1; R)+nc_n^{-2}\omega_n R^{n} A(r,1; R)\\
&\qquad+nc_n^{-2} \int_R^{r+1} \left( nA(r,1; s)-rA_r(r,1; s)\right)\omega_n s^{n-1}ds\\
                  &=-\frac{\tilde c_n}{2}S(R,r)^{\frac{n-1}{2}}(R^2+r^2-1) -\frac{c_{n-1}}{c_n} r^nR^{n} \int_{\frac{s_1(r;R)}{r}}^{1}(1-y^2)^{\frac{n-1}{2}}dy \\ 
                &\qquad+\frac{1}{c_n} R^n A(r,1;R) +\frac{1}{c_n}\underbrace{ \int_{R}^{r+1}(n A(r,1;s)-rA_r(r,1;s) ) s^{n-1}ds}_{=:\Gamma},
              \end{split}
              \end{equation}
              where we used \eqref{I_Rr_derivative} in the first equality, and used \eqref{eq:Ar} (with $s$ replaced by $R$), \eqref{eq:l&S}, the definition of $\tilde c_n$, and the fact that $s_1(r;R)=\frac{R^2+r^2-1}{2R}$ in the second equality.
                            
           Let us express $A(r,1,R)$ using its integral formulation as in \eqref{eqI_5} (with $s$ replaced by $R$), and plug it into the above equation. $V$ then becomes
                          \begin{equation}\label{eq:V2}
                \begin{split}
                 V    & = -\frac{\tilde c_n}{2}S(R,r)^{\frac{n-1}{2}}(R^2+r^2-1) +\frac{c_{n-1}}{c_n} R^{n} \int_{R-s_1(r;R)}^{1}(1-y^2)^{\frac{n-1}{2}}dy +\frac{1}{c_n} \Gamma.
           \end{split}
        \end{equation}
        Comparing \eqref{eq:V2} with \eqref{eq:v}, all we need to show is that 
        \begin{equation}\label{eq:goal_gamma}
        \Gamma = -c_{n-1}R^{n}\int_{R-s_1(r;R)}^{1}(1-y^2)^{\frac{n-1}{2}}dy.
        \end{equation}
         By \eqref{eqI_5} and \eqref{eq:Ar}, the parenthesis in the integrand in $\Gamma$ can be written as 
        \[
           \begin{split}
              nA(r,1;s)-rA_r(r,1;s) 
                            & =c_{n-1}\left( n\int_{s-s_1(r;s)}^{1}(1-y^2)^{\frac{n-1}{2}}dy- l(r;s)^{n-1}s_1(r;s) \right),
           \end{split}
        \]
        thus 
        \begin{equation}\label{gamma_temp}
           \begin{split}
              \Gamma & =c_{n-1}\underbrace{\int_{R}^{r+1} ns^{n-1}\int_{s-s_1(r;s)}^{1}(1-y^2)^{\frac{n-1}{2}}dyds}_{=: \Gamma_1} -c_{n-1}\int_{R}^{r+1}s^{n-1} l(r;s)^{n-1} s_1(r;s)ds.
           \end{split}
        \end{equation}
         We can then use integration by parts to express $\Gamma_1$ as follows:
        \[
          \begin{split}
             \Gamma_1 
                              & =\left(s^n\int_{s-s_1(r;s)}^{1}(1-y^2)^{\frac{n-1}{2}}dy\right)\bigg |_{R}^{r+1}-\int_{R}^{r+1}s^{n}\frac{\partial}{\partial s}\left(\int_{s-s_1(r;s)}^{1}(1-y^2)^{\frac{n-1}{2}}dy\right)ds\\
                               & =-R^n\int_{R-s_1(r;R)}^{1}(1-y^2)^{\frac{n-1}{2}}dy+\int_{R}^{r+1}s^n\big(1-(s-s_1(r;s))^2\big)^{\frac{n-1}{2}}\frac{\partial}{\partial s}(s-s_1(r;s))ds\\
                               & =-R^n\int_{R-s_1(r;R)}^{1}(1-y^2)^{\frac{n-1}{2}}dy+\int_{R}^{r+1}s^n l(r;s)^{n-1} \frac{s^2+r^2-1}{2s^2}ds\\
                                & =-R^n\int_{R-s_1(r;R)}^{1}(1-y^2)^{\frac{n-1}{2}}dy+\int_{R}^{r+1}s^{n-1} l(r;s)^{n-1} s_1(r;s) ds,
          \end{split}
        \]
        where we used $s_1(r;r+1)=r$ in the second equality, the fact that $l(r;s) = \sqrt{1-(s-s_1)^2}$ and \eqref{s1_comp} in the third equality.
     Finally, plugging the above expression of $\Gamma_1$ into \eqref{gamma_temp} gives \eqref{eq:goal_gamma}, thus finishes the proof.
   \end{proof}

\section{Non-uniqueness of steady states for $1<m<2$}
In this section, for $1<m<2$, given any attractive potential $W_0$ satisfying \textup{(W1)} and \textup{(W2)} with  $k>-n(m-1)$ and $k\geq -n+2$, we aim to modify its tail to construct a new attractive potential that has infinitely many radially decreasing steady states in $\mathcal{P}(\mathbb{R}^n)$. The reason for the two extra assumptions $k>-n(m-1)$ (i.e. \eqref{eq:evolution} is in the diffusion-dominated regime) and $k\geq -n+2$ (i.e. $W$ is no more singular than the Newtonian potential at the origin) is that even though the result only deals with steady states, our proof strategy however requires the dynamical solution to \eqref{eq:evolution} to exist globally in time. Under these extra assumptions, it was established in \cite{BRB} that for any initial data $\rho_0 \in \mathcal{P}(\mathbb{R}^n)\cap L^\infty(\mathbb{R}^n)$, there exists a unique global-in-time weak solution $\rho(\cdot, t)$ to \eqref{eq:evolution}, which satisfies the Energy Dissipation Inequality: for any $t>0$, we have
\begin{equation}\label{edi}
 \mathcal{E}[\rho(t)] + \int_0^t \mathcal{D}[\rho(t)] dt \leq \mathcal{E}[\rho_0],
\end{equation}
where $\mathcal{D}[\rho] = \int_{\mathbb{R}^n} \rho |\nabla(\frac{m}{m-1}\rho^{m-1}+W*\rho)|^2dx$.

\subsection{Modifying the tail}
Let us take any attractive potential $W$ satisfying \textup{(W1)} and \textup{(W2)} with $k>-n(m-1)$ and $k\geq -n+2$. Assume $W$ has a radial steady state $\rho_s$ that is supported in some $B(0,R)$. Note that no matter how we modify $W$ outside $B(0,2R)$, $\rho_s$ is still a steady state, since $\sup_{x,y\in\supp\rho_s} \text{dist}(x,y)\leq 2R$. With this in mind, we will modify $W$ as follows to obtain a new steady state. Let $\eta:\mathbb{R}\to \mathbb{R}$ be a smooth cut-off function, such that 
\begin{equation}\label{def_eta}
\eta(s) = \begin{cases}
0 & \text{ for }s\leq 0\\
1 & \text{ for }s\geq 1,
\end{cases}
\end{equation}
 and $\eta$ is strictly increasing in $(0,1)$ with $\eta' < 2$. We then modify $W(x)$ into 
\begin{equation}\label{def_w0}
W_{R,\epsilon}(x) := w_1(x)+ w_2(x),
\end{equation}
 where $w_1$ and $w_2$ are both radial, and by a slight abuse of notation we define them below as functions of the radial variable $r$. Let their derivatives be given by
\begin{equation}\label{def_w1}
w_1'(r) := W'(r) \left(1-\eta\left(\frac{r-2R}{R}\right)\right)\quad \text{ for all }r>0,
\end{equation}
and
\begin{equation}\label{def_w2}
w_2'(r) := \epsilon \eta\left(\frac{r-2R}{R}\right) \quad \text{ for all }r>0,
\end{equation}
where $\epsilon>0$ is a small constant to be determined later. Note that $w_1'(r)=W'(r)$ and $w_2'(r)\equiv 0$ for all $r\in(0,2R]$. Let us set in addition that $w_1(r)=W(r)$ and $w_2(r)\equiv 0$ in $(0,2R]$. With such definition, $W_{R,\epsilon}(r)$ is identical to $W(r)$ for $r\in [0,2R]$, and it is smooth and strictly increasing for $r\in(0,\infty)$. Note that \eqref{def_w1} and \eqref{def_w2} give that $w_1$ is constant for $r\in [3R,+\infty)$, and $w_2'(r) \equiv \epsilon$ for $r>3R$. A sketch of $W_{R,\epsilon}$ is given in Figure~\ref{W_fig}.

\begin{figure}[h!]
\begin{center}
\includegraphics[scale=1.1]{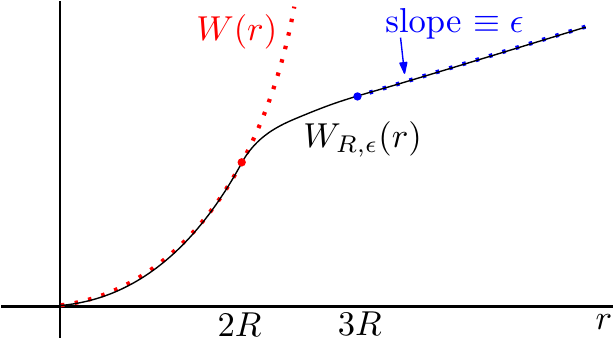}
\caption{Illustration of $W$ (red dotted curve) and the modified $W_{R,\epsilon}$ (black curve).\label{W_fig}}
\end{center}
\end{figure}

%
%
%

\subsection{Heuristics for non-uniqueness}
For the interaction potential $W_{R,\epsilon}$ defined in \eqref{def_w0}, our goal is to show that if $\epsilon>0$ is sufficiently small, then there exists another steady state which is ``flatter'' than $\rho_s$. Before the rigorous proof in the next subsection, let us first outline some formal heuristics here.

To see the motivation of modifying $W$ into $W_{R,\epsilon}$, let us first consider the extreme case $\epsilon=0$, where $ W_{R,\epsilon} = w_1$ is constant in $[3R,\infty)$. In this case, we can subtract $W_{R,\epsilon} $ by a constant such that $W_{R,\epsilon}(3R)=0$. In this way, $W_{R,\epsilon} \in L^1(\mathbb{R}^n)$ since it is compactly supported in $B(0,3R)$. Note that the solution to \eqref{eq:evolution} remains unaffected as we add/subtract a constant to the interaction potential.

For an $L^1$ potential, the following formal scaling argument suggests that it is energy favorable for a sufficiently ``flat'' initial data to continue spreading. To see this, let us take any continuous function $\rho$ with finite mass, consider its dilation $\rho_\lambda(x) := \lambda^n \rho(\lambda x)$, and check whether $\mathcal{E}[\rho_\lambda]$ is increasing or decreasing in $\lambda$ as $\lambda \to 0^+$. Under the dilation, we have (see \cite[Section 2.3.1]{CCY} for a derivation)
\[
\mathcal{E}[\rho_\lambda] =\lambda^{(m-1)n} \int_{\mathbb{R}^n} \frac{1}{m-1}\rho^m dx  + \frac{\lambda^n}{2} \left(\int_{\mathbb{R}^n} W_{R,\epsilon} dx\right) \int_{\mathbb{R}^n} \rho^2 dx + o(\lambda^n).
\]
When $m<2$, since $\lambda^{(m-1)n} \gg \lambda^n$ for $\lambda \ll 1$, we have that $\mathcal{E}[\rho_\lambda]$ is increasing in $\lambda$ for $\lambda \ll 1$, meaning that it is energy favorable for a sufficiently flat solution to become even flatter. As a result, for $\epsilon=0$ we expect that sufficiently flat initial data should spread to infinity as $t\to\infty$. 

Instead of setting $\epsilon=0$, we will actually work with a sufficiently small $0<\epsilon\ll 1$. On the one hand, using the smallness of $\epsilon$, we will show that solutions with sufficiently flat initial data will remain flat for all time, thus $\rho(\cdot,t)$ cannot return to $\rho_s$ as $t\to\infty$. On the other hand, for any $\epsilon>0$, it is not possible for $\rho(\cdot,t)$ to spread to infinity either: since $W(r)$ grows linearly in $r$ for large $r$, we will show that the first moment of $\rho(\cdot,t)$ is uniformly bounded in time. Putting these two pieces together will give us existence of another steady state that is flatter than $\rho_s$.

\subsection{Proving non-uniqueness by tracking the $L^{3-m}$ norm}

In the following lemma, we show that when $\epsilon$ is sufficiently small in the modified potential $W_{R,\epsilon}$, solutions with sufficiently flat initial data will remain flat for all time. Here we control the ``flatness'' of a solution by tracking the evolution of its $L^{3-m}$ norm. We choose this norm due to the technical reason that the degenerate diffusion term gives exactly $-c(m)\|\nabla \rho\|_2^2$, making the computation easier. 

\begin{lemma}\label{lem_1}
Let $1<m<2$, and let $W$ satisfy \textup{(W1)} and \textup{(W2)} with $k>-n(m-1)$ and $k\geq -n+2$. Given any $R>0, a>0$, there exists some $\delta_0 \in (0,a)$ and $\epsilon \in (0,1)$ depending on $W, R, a$ and $n$, such that the solution $\rho(x,t)$ to \eqref{eq:evolution} with interaction potential $W_{R,\epsilon}$ defined in \eqref{def_w0}--\eqref{def_w2} satisfies the following:  For any initial data with $\rho_0\in \mathcal{P}(\mathbb{R}^n) \cap L^\infty(\mathbb{R}^n)$ and $\|\rho_0\|_{3-m}<\delta_0$, the solution $\rho(\cdot,t)$ satisfies   \begin{equation}
      \|\rho(\cdot, t)\|_{3-m}< \delta_0 \quad\text{ for all } t>0.
   \end{equation}
\end{lemma} 
\begin{proof}
The evolution of $\int_{\mathbb{R}^n} \rho(x,t)^{3-m}dx$ in time can be computed as follows.  Multiplying \eqref{eq:evolution} by $\rho^{2-m}$ and integrating in $\mathbb{R}^n$, we have 
   \[
      \int_{\mathbb{R}^n}\rho^{2-m}\rho_tdx=\int_{\mathbb{R}^n} \rho^{2-m}\Delta \rho^mdx+\int_{\mathbb{R}^n} \rho^{2-m}\nabla \cdot (\rho \nabla (W_{R,\epsilon}*\rho))dx.
   \]
   Integration by parts gives 
   \begin{equation}\label{eq1_1}
     \begin{split}
      \frac{1}{3-m}\frac{d}{dt} \int_{\mathbb{R}^n}\rho^{3-m}dx & =-\int_{\mathbb{R}^n}\nabla \rho^{2-m}\cdot \nabla \rho^mdx-\int_{\mathbb{R}^n}\nabla \rho^{2-m}\cdot (\rho \nabla (W_{R,\epsilon}*\rho)) dx\\
                                                                                                & =-(2-m)m \|\nabla \rho\|_2^2-(2-m) \underbrace{\int_{\mathbb{R}^n}\rho^{2-m}\nabla \rho \cdot  \nabla (W_{R,\epsilon}*\rho) dx}_{=: I(t)},\\
      \end{split}
   \end{equation}
   \vspace*{-0.3cm}
where 
\[
   \begin{split}
      I(t) 
        & =\int_{\mathbb{R}^n}\rho^{2-m}\nabla \rho \cdot  \nabla (w_1*\rho) dx+\int_{\mathbb{R}^n}\rho^{2-m}\nabla \rho \cdot  \nabla (w_2*\rho) dx =:I_1(t)+I_2(t).
   \end{split}
\]

Before we continue to estimate $I_1$ and $I_2$, let us first point out some properties of the functions $w_1$ and $w_2$ defined in \eqref{def_w1} and \eqref{def_w2}. Without loss of generality, we can subtract $w_1$ by a constant (which does not change the solution $\rho(\cdot,t)$) such that $w_1(3R)=0$. In this way, we have $\supp w_1 \subset B(0,3R)$. Since $w_1$ also satisfies (W2) near the origin, 
\begin{equation}\label{eqw_1}
   \|w_1\|_q \le C(W,R)\quad\text{ for all  $q\geq 1$ such that $qk>-n$},
   \end{equation} where $C$ only depends on $W$ and $R$.
As for $w_2$, we have 
\begin{equation}\label{estimate_w2}
\begin{split}
\|\Delta w_2\|_\infty &\leq \sup_{r\geq 2R} \left(w_2''(r) + \frac{n-1}{r} w_2'(r)\right)\\
&\leq \epsilon\left(\frac{1}{2R} \|\eta'\|_\infty + \frac{n-1}{R} \|\eta\|_\infty \right) \leq \frac{n}{R} \epsilon,
\end{split}
\end{equation}
where we used $\|\eta'\|_\infty\leq 2$ and $\|\eta\|_\infty\leq 1$ in the last inequality. 

Now let us estimate $I_1$. By H\"older's inequality (for three functions), we have
\begin{equation}\label{i1_estimate1}
\begin{split}
|I_1|               & \le \int_{\mathbb{R}^n}\rho^{2-m}|\nabla \rho|| \nabla( w_1* \rho)| \,dx\\[0.1cm]
&\leq \|\rho^{2-m}\|_{\frac{3-m}{2-m}} \|\nabla \rho\|_2 \|w_1 * \nabla \rho\|_p\\[0.2cm]
&\leq \|\rho\|_{3-m}^{2-m} \|\nabla \rho\|_2 \|w_1 * \nabla \rho\|_p,
\end{split}
\end{equation}
where $p$ is such that $\frac{2-m}{3-m} + \frac{1}{2} + \frac{1}{p} = 1$, which leads to $p = \frac{2(3-m)}{m-1}$. Note that our assumption $m\in (1,2)$ gives that $p>2$, thus Young's inequality and \eqref{eqw_1} yield that
%
\begin{equation}\label{tempasdf}
   \| w_1* \nabla\rho\|_p\le C\|w_1\|_q\| \nabla \rho\|_2, 
\end{equation}
where $q$ is such that $\frac{1}{q}+\frac{1}{2}=\frac{1}{p}+1$, that is, $q=3-m$. Let us check that $qk>-n$: using the assumption $k>-n(m-1)$ we have $qk > -n(m-1)(3-m) > -n$, where the last inequality follows from $m\in(1,2)$. This allows us to use \eqref{eqw_1} to control $\|w_1\|_q$, thus \eqref{tempasdf} becomes \[
 \| w_1* \nabla\rho\|_p
 \le C(W,R) \| \nabla \rho\|_2,
\]
 and plugging this into \eqref{i1_estimate1} gives 
\begin{equation}\label{I1_temp}
   |I_1|\le C(W,R) \|\rho\|_{3-m}^{2-m}\|\nabla \rho\|^2_{2}.
\end{equation}

In the rest of the proof, we fix $\delta_0 \in (0,a)$ such  that $C(W, R)\delta_0^{2-m}<\frac{m}{4}$, with $C(W,R)$ as in \eqref{I1_temp}. Since $m\in (1,2)$, there indeed exists such a $\delta_0$.  Our goal is to show that for sufficiently small $\epsilon>0$  (to be fixed later),   all solutions with initial data $\|\rho_0\|_{3-m} < \delta_0$ satisfy that $\|\rho(t)\|_{3-m}< \delta_0$ for all $t>0$. To show this, it suffices to look at the first time that $\|\rho(t)\|_{3-m}=\delta_0$ (call it $t_1$), and aim to show that $\frac{d}{dt}\int \rho(x,t)^{3-m}dx \big |_{t=t_1}< 0$.

From our choice of $\delta_0$ and the definition of $t_1$, at $t=t_1$ we can control $I_1$ as 
\begin{equation}\label{eqI_1}
   |I_1(t_1)|\le C(W, R)\delta_0^{2-m}\|\nabla \rho (t_1)\|^2_{2} \le \frac{m}{4}\|\nabla \rho(t_1)\|^2_{2}.
\end{equation}
%
Next we move on to $I_2$. We first rewrite $I_2$ as
\[
   \begin{split}
      I_2  
             =\frac{1}{3-m}\int_{\mathbb{R}^n}\nabla \rho^{3-m} \cdot  (\nabla w_2*\rho) dx =-\frac{1}{3-m}\int_{\mathbb{R}^n} \rho^{3-m}  (\Delta w_2*\rho) dx.
   \end{split}
\]
Let us apply the Young's inequality to $I_2$, and combine it with \eqref{estimate_w2} and $\|\rho\|_1=1$:
\[
  \begin{split}
   |I_2| 
           & \le \frac{1}{3-m}\|\rho\|^{3-m}_{3-m}\|\Delta w_2\|_{\infty}\|\rho\|_{1} \le \frac{C(n,R)\epsilon}{3-m}\|\rho\|^{3-m}_{3-m}.
   \end{split}
\]
By the Gagliardo--Nirenberg inequality, we have 
\[
   \|\rho\|_{3-m}\le C(n)\|\nabla \rho\|_{2}^{\theta}\|\rho\|_{1}^{1-\theta}=C(n)\|\nabla \rho\|_{2}^{\theta},
\]
where $\theta=\frac{2n(2-m)}{(3-m)(2+n)}\in (0,1)$. Combining this inequality with the assumption that $\|\rho(t_1)\|_{3-m}=\delta_0$, we have 
\[
  \begin{split}
   |I_2(t_1)| & \le \frac{C(n,R)\epsilon}{3-m}\|\rho(t_1)\|^{3-m-\frac{2}{\theta}}_{3-m} \|\rho(t_1)\|^{\frac{2}{\theta}}_{3-m}
                   \le \frac{C(n,R)\epsilon\delta_0^{3-m-\frac{2}{\theta}}}{3-m}\|\nabla \rho (t_1)\|^{2}_{2}.
   \end{split}
\]
Now we can choose $\epsilon>0$ sufficiently small such that $\frac{C(n,R)\epsilon\delta_0^{3-m-\frac{2}{\theta}}}{3-m}<\frac{m}{4}$, leading to 
\begin{equation}\label{eqI_2}
   |I_2(t_1)|\le \frac{m}{4}\|\nabla \rho (t_1)\|^{2}_{2}.
\end{equation}
Finally, combining \eqref{eqI_1} and \eqref{eqI_2}, we have $| I(t_1) |\le \frac{m}{2}\|\nabla \rho (t_1)\|^{2}_{2}$. Apply this to (\ref{eq1_1}) gives
\[
 \frac{d}{dt} \int_{\mathbb{R}^n}\rho^{3-m}(x,t)dx\,\Big|_{t=t_1}\le -\frac{m(3-m)(2-m)}{2}\|\nabla \rho (t_1)\|^{2}_{2}<0,
\]
which contradicts the definition of $t_1$, thus finishing the proof of the lemma.
%
\end{proof}

In the next lemma, we show that the modified potential $W_{R,\epsilon}$ in Lemma~\ref{lem_1} indeed leads to a radial steady state $\rho_s$ with $\|\rho_s\|_{3-m} \leq \delta_0$. As we will see later in the proof of the main theorem, we will apply Lemma~\ref{lemma2} iteratively to construct an interaction potential with infinitely many steady states.

\begin{lemma}\label{lemma2}
Let $1<m<2$, and let $W$ satisfy \textup{(W1)} and \textup{(W2)} with $k>-n(m-1)$ and $k\geq -n+2$. Given any $R>0, a>0$, let $\delta_0 \in (0,a)$ and $\epsilon \in (0,1)$ be as given in Lemma~\ref{lem_1}. Then the equation \eqref{eq:evolution} with interaction potential $W_{R,\epsilon}$ has a compactly supported radially decreasing steady state $\rho_s\in \mathcal{P}(\mathbb{R}^n) \cap L^\infty(\mathbb{R}^n)$ with  $\|\rho_s\|_{3-m} \leq \delta_0$.
\end{lemma}

\begin{proof}Let $\delta_0 \in (0,a)$ and $\epsilon>0$ be as given in Lemma~\ref{lem_1},
 and take any radially symmetric initial data $\rho_0 \in \mathcal{P}(\mathbb{R}^n) \cap L^\infty(\mathbb{R}^n) $ with $\|\rho_0\|_{3-m} < \delta_0$. It then follows from Lemma~\ref{lem_1} that the solution $\rho(\cdot,t)$ to \eqref{eq:evolution} with potential $W_{R,\epsilon}$ satisfies $\|\rho(t)\|_{3-m} < \delta_0$ for all $t\geq 0$. 

Since the interaction potential $W_{R,\epsilon}$ has linear growth for large $|x|$, we claim that the first moment of $\rho(\cdot,t)$ is uniformly bounded in time. To see this, first note that the interaction energy itself is uniformly bounded in time:  since $\mathcal{E}[\rho(t)] \leq \mathcal{E}[\rho_0]$ for all $t\geq 0$, using the non-negativity of $\frac{1}{m-1}\int \rho^m(x,t) dx$, we have that 
\[
\frac{1}{2}\iint_{\mathbb{R}^n \times \mathbb{R}^n} \rho(x,t) \rho(y,t) W_{R,\epsilon}(x-y) dxdy \leq \mathcal{E}[\rho_0]. 
\]
By definition of $W_{R,\epsilon}$ in \eqref{def_w0}--\eqref{def_w2}, we have that $W_{R,\epsilon}(x-y) dxdy \geq w_1(3R)+\epsilon(|x|-3R)_+ \geq \epsilon|x| +w_1(3R)- 3R\epsilon$, and applying it to the previous inequality gives
\[
\epsilon \iint_{\mathbb{R}^n \times \mathbb{R}^n} \rho(x,t) \rho(y,t) |x-y| dxdy \leq  3R\epsilon-w_1(3R) + 2 \mathcal{E}[\rho_0]. 
\] 
Finally, to relate the interaction energy with the first moment, by \cite[Lemma 2.8]{CDP} and the fact that $\rho_0 \in \mathcal{P}(\mathbb{R}^n)$, there exists a universal constant $C$, such that 
\begin{equation}\label{unif_m1}
\int_{\mathbb{R}^n} \rho(x,t) |x| dx \leq C \iint_{\mathbb{R}^n \times \mathbb{R}^n} \rho(x,t) \rho(y,t) |x-y| dxdy \leq \frac{C}{\epsilon} ( 3R\epsilon-w_1(3R)+ 2 \mathcal{E}[\rho_0]) =: C_1,
\end{equation}
finishing the proof of the claim. 

In addition, we have $\|\rho\|_{L^\infty(\mathbb{R}^n\times (0,\infty))} \leq C_2$ for some $C_2<\infty$. This can be done by the same bootstrap iterative argument in \cite[Theorem 1.1]{Liu2016} on the $L^p$ norm for a sequence of $p\to\infty$. Even though \cite{Liu2016} focused on the Newtonian potential in the diffusion-dominated regime, the same proof can be applied to our potential $W_{R,\epsilon}$ as well. Note that $W_{R,\epsilon}$ is identical to $W$ in a small neighborhood of the origin, thus it satisfies (W2) with $k>-n(m-1)$ and $k\geq -n+2$. The first assumption ensures that the equation is in the diffusion-dominated regime, and the second assumption ensures that it is no more singular than the Newtonian potential.

In the rest of the proof we aim to show that there exists a steady state $\rho_s \in \mathcal{P}(\mathbb{R}^n)\cap L^\infty(\mathbb{R}^n)$ with $\|\rho_s\|_{3-m}\leq \delta_0$ and $\int_{\mathbb{R}^n} \rho_s |x| dx \leq C_1$. Roughly speaking, the idea is that ``uniform-in-time bounds on $\rho(t)$ + tightness of $\{\rho(t)\}_{t>0}$ $\Rightarrow$ existence of a steady state''. The proof is mostly identical to \cite[Theorem 4.12]{CHVY}, except the following two differences. The first is that in \cite{CHVY}, tightness of $\{\rho(t)\}_{t>0}$ comes from a uniform-in-time bound of the second moment, which is obtained by using some special properties of Newtonian potential in 2D; whereas in this proof tightness comes from the uniform-in-time bound of the first moment \eqref{unif_m1}. The second difference is that in this proof we will only obtain the existence of a steady state $\rho_s$ with the desired bound on the $L^{3-m}$ norm and the first moment, but without any convergence result $\rho(t)\to \rho_s$ as $t\to\infty$; whereas in \cite{CHVY} one can obtain convergence as $t\to\infty$ due to uniqueness of steady state in that setting.

For the sake of completeness, we briefly sketch the main steps below. By the Energy Dissipation Inequality \eqref{edi}, the global weak solution $\rho$ satisfies the entropy inequality
\[
\lim_{t\to\infty} \mathcal{E}[\rho(t)] + \int_0^\infty \mathcal{D}[\rho(t)] dt \leq \mathcal{E}[\rho_0],
\]
where $\mathcal{D}[\rho(t)] := \int_{\mathbb{R}^n} \rho |\nabla h[\rho]|^2 dx$, with  $h[\rho] := \frac{m}{m-1}\rho^{m-1} + \rho*W_{R,\epsilon}$.
Since $W_{R,\epsilon}$ satisfies (W2) with $k>-n(m-1)$, we have that  $\mathcal{E}[\rho(t)]$ is bounded below. Thus $\int_0^\infty \mathcal{D}[\rho(t)] dx < \infty$, implying 
\begin{equation}\label{d_temp}
\lim_{t\to\infty} \int_t^\infty \mathcal{D}[\rho(s)]ds = 0.
\end{equation} For any diverging time sequence $\{t_j\}_{j=1}^\infty$, let us consider the sequence of functions $\{\rho_j\}_{j=1}^\infty$ given by 
\[
\rho_j(\cdot,t) := \rho(\cdot, t+t_j) \quad\text{in }  \mathbb{R}^n \times (0,T),
\] 
which is well-defined since $\rho(\cdot,t)$ is a global-in-time weak solution. Using \eqref{d_temp}, we have
\[
\int_0^T \mathcal{D}[\rho_j(t)] dt = \int_{t_j}^{t_j+T} \mathcal{D}[\rho(t)] dt \leq \int_{t_j}^{\infty} \mathcal{D}[\rho(t)] dt \to 0 \quad\text{ as } j\to\infty,
\]
thus equivalently, using the fact that $\mathcal{D}[\rho] = \int_{\mathbb{R}^n}  \rho |\nabla h[\rho]|^2 dx$, we have that
\begin{equation}\label{l2_temp}
\|\sqrt{\rho_j} |\nabla h[\rho_j]|\|_{L^2(\mathbb{R}^n\times(0,T))} \to 0 \quad\text{ as }j\to\infty.
\end{equation}

By \cite[Lemma 4.17]{CHVY}, there exists a function $ \bar\rho \in L^1(\mathbb{R}^n \times (0,T)) \cap L^m(\mathbb{R}^n \times (0,T))$ and a subsequence of $\{\rho_j\}$ (which we still denote by $\{\rho_j\}$), such that
\begin{equation}\label{weak_conv}
\rho_j(x,t) \rightharpoonup \bar\rho(x,t) \quad \text{ in } L^1(\mathbb{R}^n \times (0,T)) \cap L^m(\mathbb{R}^n \times (0,T))
\end{equation}
as $j\to \infty$. Our goal is to show that the function $\bar\rho(x,t)$ is in fact independent of time (call it $\rho_s(x)$), and it is a steady state of \eqref{eq:evolution} with $\|\rho_s\|_{3-m}\leq \delta_0$ and $\int \rho_s |x| dx \leq C_1$.

First, note that $\|\rho_j\|_{L^\infty(\mathbb{R}^n\times(0,T))}\leq C_2$ for all $j$, which follows from the bound $\|\rho\|_{L^\infty(\mathbb{R}^n\times (0,\infty))}\leq C_2$. In addition, recall that we have $\|\rho(t)\|_{3-m}\leq \delta_0$ and $\|\rho(t) |x|\|_1\leq C_1 $ for all $t\geq 0$. Combining these with the weak convergence in \eqref{weak_conv}, we have that 
\begin{equation}\label{rho_bd1}
\|\bar\rho\|_{L^\infty(0,T; L^\infty(\mathbb{R}^n))} \leq C_2,\quad \|\bar\rho\|_{L^\infty(0,T; L^{3-m}(\mathbb{R}^n))} \leq \delta_0, \quad \big\|\bar\rho |x| \big\|_{L^\infty(0,T; L^{1}(\mathbb{R}^n))} \leq C_1.
\end{equation}


Proceeding as in Lemma 4.18, 4.19 and 4.20 of \cite{CHVY}, along a further subsequence (which we still denote by $\{\rho_j\}$), we have 
\begin{equation}\label{conv_temp2}
\begin{split}
&\rho_j  \to \bar \rho \quad\text{ in }L^q(\mathbb{R}^n \times (0,T)) \quad\text{ for any }1\leq q < \infty,\\
&\rho_j^p  \to \bar\rho^p \quad\text{ in }L^2(0,T; H^1(\mathbb{R}^n)) \quad\text{ for any }m-\frac{1}{2} \leq p < \infty,\\
&\sqrt{\rho_j} \nabla h[\rho_j] \rightharpoonup \sqrt{\bar\rho} \nabla h[\bar\rho] \quad\text{ in } L^2(\mathbb{R}^n \times (0,T);\mathbb{R}^n).
\end{split}
\end{equation}
The above convergence results yield that $\bar\rho$ is a weak distributional solution to \eqref{eq:evolution} with test functions in $L^2(0,T; H^1(\mathbb{R}^n))$. In addition, since the $L^2$ norm is weakly lower semicontinuous, combining the third weak convergence result in \eqref{conv_temp2} with \eqref{l2_temp} gives that 
\[
\bar\rho|\nabla h[\bar\rho]|^2 = 0 \quad\text{ a.e. in }\mathbb{R}^n \times(0,T),
\]
implying that $\sqrt{\bar\rho}\nabla h[\bar\rho] = 0$ a.e. in $\mathbb{R}^n \times(0,T)$.
In addition, due to the convergence results in \eqref{conv_temp2} we know that $\bar\rho$ is a weak distributional solution to \eqref{eq:evolution} with test functions in $L^2(0,T; H^1(\mathbb{R}^n))$. This gives that $\partial_t \bar \rho = 0 $ in $L^2(0,T;H^{-1}(\mathbb{R}^2))$, thus $\bar\rho(t,x) \equiv \rho_s(x)$ is independent of time, and is a (weak) steady state. Since $\rho(\cdot,t)$ is a radial solution, we know that $\rho_s$ is radial as well, and thus must be radially decreasing by Theorem 2.2 of \cite{CHVY}.
From \eqref{rho_bd1}, we have that 
\[
\|\rho_s\|_\infty \leq C_2, \quad \|\rho_s\|_{3-m}\leq \delta_0, \quad\text{and }\int_{\mathbb{R}^n} \rho_s(x) |x| dx \leq C_1.
\]
From  \cite[Lemma 2.22]{CHVY} and using the uniform first moment of $\rho_j$, we have that no mass is lost as we take the $j\to\infty$ limit, that is, $\rho_s \in \mathcal{P}(\mathbb{R}^n)$. 

Finally, note that the potential $W_{R,\epsilon}$ satisfies (W1)--(W4): It satisfies (W1) and (W2) since $W_{R,\epsilon} = W$ in some neighborhood of the origin. It satisfies (W3) and (W4) since $W'(r)$ is a constant for all sufficiently large $r$. In addition, $\lim_{r\to\infty}W_{R,\epsilon}=\infty$, which allows us to apply Lemma~\ref{lemma_stat}(b,d) to conclude that $\rho_s$ is radially decreasing and compactly supported.
\end{proof}

Now we are ready to prove the main theorem for the non-uniqueness result. 


\begin{proof}[\textbf{\textup{Proof of Theorem~\ref{thm:nonuniqueness}}}]

 The potential $\widetilde W$ is constructed by an inductive argument. For the base step, using the given potential $W_0$ and $R_0>0$, let us take any $a>0$ (e.g. $a=1$),  and let $W_1:=W_{R_0,\epsilon}$ be the interaction potential given in Lemma~\ref{lemma2}. From the way we modify the tails in \eqref{def_w0}--\eqref{def_w2} and the assumptions on $W_0$, we have that $W_1 \in C^\infty(\mathbb{R}^n\setminus\{0\}) \cap W^{1,\infty}(B(0,R_0)^c)$, with 
 \[
 \|\nabla W_1\|_{L^\infty(B(0,R_0)^c)} \leq \max\{\|\nabla W_0\|_{L^\infty(B(0,3R_0)\setminus B(0,R_0))}, 1\}.
 \] By Lemma~\ref{lemma2}, equation \eqref{eq:evolution} with potential $W_1$ has a compactly supported radially decreasing steady state $\rho_s^1 \in \mathcal{P}(\mathbb{R}^n) \cap L^\infty(\mathbb{R}^n)$, and we denote its support by $B(0,R_1)$. (See the red curves in Figure~\ref{fig_iter}).

\begin{figure}[h!]
\begin{center}
\includegraphics[scale=1]{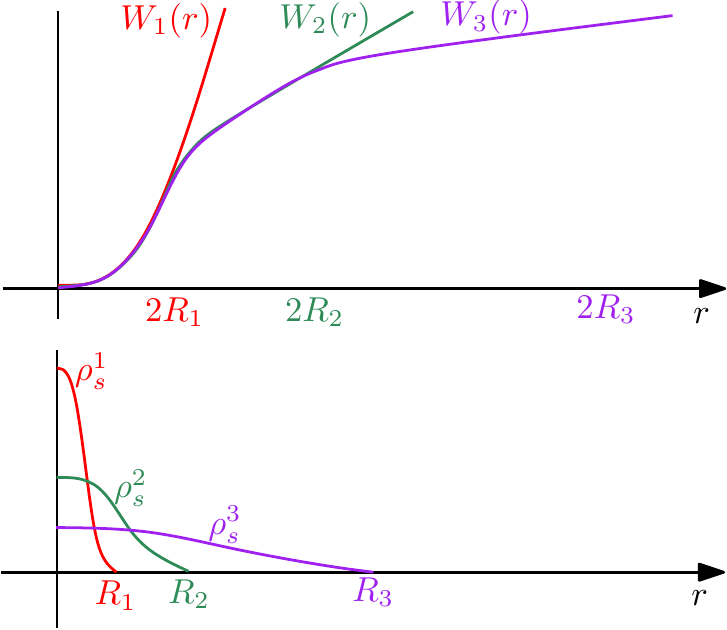}
\caption{Illustration of the iterative construction of $\widetilde W$, and the sequence of steady states $\{\rho_s^i\}$ it leads to.\label{fig_iter}}
\end{center}
\end{figure}

Next we do the inductive step. For $N\geq 1$, assume that we have already obtained an attractive potential $W_N \in  C^\infty(\mathbb{R}^n \setminus\{0\}) \cap W^{1,\infty}(B(0,R_0)^c)$  with a sequence of radial steady states $\{\rho_s^i\}_{i=1}^N$, where $\rho_s^i  \in \mathcal{P}(\mathbb{R}^n) \cap L^\infty(\mathbb{R}^n)$ is compactly supported, and its support is denoted by $B(0,R_i)$ respectively for $i=1,\dots,N$. In addition, assume that $W_N$ satisfies the estimate\begin{equation}\label{estimate_temp}
\|\nabla W_{N}\|_{L^\infty(B(0,R_0)^c)} \leq \max\{\|\nabla W_0\|_{L^\infty(B(0,3R_0)\setminus B(0,R_0))}, 1\}.
\end{equation}
Our goal is to modify $W_N$ into $W_{N+1} \in  C^\infty(\mathbb{R}^n \setminus\{0\}) \cap W^{1,\infty}(B(0,R_0)^c)$, such that it has an extra radial steady state $\rho_s^{N+1}$ in addition to $\{\rho_s^i\}_{i=1}^N$, and it satisfies the estimate \eqref{estimate_temp} with $N$ replaced by $N+1$. 

 We construct $W_{N+1}$ as follows. Let
\begin{equation}\label{wn}
W := W_N,\quad  a:= \frac{1}{2}\min_{1\leq i\leq N} \left\|\rho_s^i\right\|_{3-m}, \quad R:= \max\{R_1, \dots, R_N, 2^N\}.
\end{equation}
For such $W$, $a$ and $R$, let $\delta_0 \in (0,a)$ and $\epsilon \in (0,1)$ be as given in Lemma~\ref{lemma2}, and let $W_{N+1}:=W_{R,\epsilon}$ be the modified potential.  By Lemma~\ref{lem_1} and Lemma~\ref{lemma2}, equation \eqref{eq:evolution} with potential $W_{N+1}$ has a radially decreasing steady state $\rho_s^{N+1}\in \mathcal{P}(\mathbb{R}^n) \cap L^\infty(\mathbb{R}^n)$ with  $\|\rho_s^{N+1}\|_{3-m} \leq a$. Note that $\rho_s^{N+1}$ must be different from all $\{\rho_s^i\}_{i=1}^N$, since its $L^{3-m}$ norm is less than a half of each of them by definition of $a$. In addition, since $W_{N+1}$ agrees with $W_N$ in $B(0,2R)$, where $R\geq R_i$ for $1\leq i\leq N$, we have that $\{\rho_s^i\}_{i=1}^N$ are still steady states for the potential $W_{N+1}$. As a result, the potential $W_{N+1}$ has at least $N+1$ radial steady states $\{\rho_s^i\}_{i=1}^{N+1}$. Moreover, from the way we modify the tails in \eqref{def_w0}--\eqref{def_w2}, and using the fact that $\epsilon\in(0,1)$, we have $W_{N+1} \in C^\infty(\mathbb{R}^n\setminus\{0\}) \cap W^{1,\infty}(B(0,R_0)^c)$, with the estimate 
\[
\|\nabla W_{N+1}\|_{L^\infty(B(0,R_0)^c)} \leq \max\{\|\nabla W_{N}\|_{L^\infty(B(0,R_0)^c)}, 1 \} \leq \max\{\|\nabla W_0\|_{L^\infty(B(0,3R_0)\setminus B(0,R_0))}, 1\}.
\]
where we used the induction assumption \eqref{estimate_temp} in the second inequality.

Finally, to find an attractive potential $\widetilde W$ with infinitely many steady states, let it be given by 
\begin{equation}\label{lim_w}
\widetilde W(x) := \lim_{N\to\infty} W_N(x).
\end{equation} First note that this limit exists pointwise, since for each $N \geq 1$ we have $W_i = W_N$  for all $i\geq N$ in $B(0,2R)$, where $R\geq 2^N$ is given by \eqref{wn}.  This also gives that $\widetilde W \in C^\infty(\mathbb{R}^n\setminus\{0\}) \cap W^{1,\infty}(B(0,R_0)^c)$, and it satisfies the estimate \eqref{estimate_temp} with $W_N$ replaced by $\widetilde W$.  In addition, note that $\widetilde W = W_N$ in $B(0,2R)$ with $R$ given by \eqref{wn}, where we have $R>R_N$. This implies $\rho_s^N$ is a steady state of \eqref{eq:evolution} with potential $\widetilde W$. Since $N$ is arbitrary, we have that  $\widetilde W$ leads to an infinite sequence of radially decreasing steady states $\{\rho_s^i\}_{i=1}^\infty$. 
\end{proof}

\begin{appendix}
\section{Fractional regularity}
\begin{lemma}\label{lem:app}
Assume that $W$ satisfies \textup{(W1)} and \textup{(W2)}. Then for any radial bump function $\phi\in C^\infty_c(\mathbb{R}^n)$, we have
\begin{equation}\label{a1temp}
(-\Delta)^{s}(W\phi)\in L^1(\R^n)\qquad\mbox{for any $0<s<\min\left\{\frac{n+k}{2},1 \right\}$.}
\end{equation}
\end{lemma}
\begin{proof}

Since $W$ satisfies (W2) (if $k>0$, without loss of generality we can assume $W(0)=0$), there exists some finite $C_0>0$ such that for any $r\in(0,1)$ we have
\begin{equation}\label{temp_w}
|(W\phi)'(r)|\le C_0 r^{k-1}\qquad\mbox{and}\qquad |W\phi(r)|\le C_0 r^k,
\end{equation}
where in the special case $k=0$ we replace $C_0 r^k$ in the second inequality by $C_0 r^{-\delta}$ for any $\delta>0$. In fact, since $W\phi$ is compactly supported, we can modify $C_0$ into another constant (still denote it by $C_0$) such that \eqref{temp_w} holds for all $r>0$.

Let $\eta$ be a standard mollifier supported in $B_1$ and $\eta_\epsilon(x) := \epsilon^{-n} \eta(\epsilon^{-1} x)$ be its dilation. For $\epsilon>0$, we consider a regularization 
$f_\epsilon := (W\phi)*\eta_\epsilon$. Note that for all  $\epsilon \in (0,1)$ that is sufficiently small, $f_\epsilon$ satisfies the inequality \eqref{temp_w} for all $r>0$ with the constant $C_0$ replaced by $2C_0$, and $\supp f_\epsilon$ is also uniformly bounded in some $B_R$ for all small $\epsilon$. 

We will show that for any $0<s<\min\left\{\frac{n+k}{2}, 1 \right\}$ and all sufficiently small $\epsilon>0$, we have
$$
\|(-\Delta)^s f_\epsilon \|_{L^1}\le C\left(n,s,k,C_0, R\right),
$$
where the right hand side is independent of $\epsilon$. The result then follows from taking $\epsilon\to 0$, lower semi-continuity and Sobolev embeddings. For notational convenience, we denote $f_\epsilon$ by $f$ from now on.

For smooth compactly supported radial functions $f$, we can use the formula for the fractional Laplacian in \cite[Theorem 1.1]{ferrari2012radial}: given $s\in(0,1)$ and $r=|x|>0$ for $x \in \mathbb{R}^n$, we have 
 \begin{equation}\label{frac_lap}\begin{split}
(-\Delta)^sf(r)= c_{s,n} \, r^{-2s}\int_1^{+\infty}\left(f(r)-f(r t)+\left(f(r)-f\left(\frac{r}{t}\right)\right)t^{-n+2s}\right)t( t^2-1
  )^{-1-2s} H(t)dt,
\end{split}
\end{equation}
where
$$
H(t)=2\pi\alpha_n\int_{0}^{\pi}(\sin \theta)^{^{n-2}} \frac{ \left(  \sqrt{t^2-\sin^2\theta}+\cos\theta\right)^{1+2s}}{\sqrt{t^2-\sin^2 \theta}} \, d\theta, \quad t \ge 1, \quad 
\alpha_n= \frac{\pi^{\frac{n-3}{2}}} { \Gamma ( \frac{n-1}{2} )}.
$$
Analyzing the behavior of $H$, there exists $c,C>0$ depending on $n$ and $s$, such that $ct^{-2s}\leq H(t)\leq Ct^{-2s}$ for all $t\geq 1$. Using this inequality,  \eqref{frac_lap} becomes
\begin{equation}\label{frac bound}
|(-\Delta)^s f(r)|\le c_{s,n} r^{-2s} \int_1^\infty \left(|f(r)-f(rt)|+\frac{\left|f(r)-f(r/t)\right|}{t^{n-2s}}\right) \frac{dt}{(t-1)^{1+2s}}.    
\end{equation}
Let us consider the case $r<1$ first; the case $r\geq 1$ will be treated at the end of this proof. For $r\in (0,1)$, we split the integral \eqref{frac bound} for $t\in(1,2]$ and $t\in[2,\infty)$. For $t\in[1,2]$ and $r\in (0,1)$, since $f$ satisfies \eqref{temp_w} with $C_0$ replaced by $2C_0$, we have the following:
$$
|f(r)-f(rt)|\le \sup_{\xi\in[r,rt]}|f'(\xi)|r(t-1)\le C_1 r^k (t-1),
$$
and
$$
\left|f(r)-f\left(\frac{r}{t}\right)\right|t^{-n+2s} \le  \sup_{\xi \in[r/t,r]}|f'(\xi)| t^{-n+2s} (r-\frac{r}{t}) \le \sup_{\xi \in[r/t,r]}|f'(\xi)|  r\frac{(t-1)}{t^{n+1-2s}}\le C_1 r^k (t-1).
$$
For $t\in[2,\infty)$ and $r\in (0,1)$, since $f$ satisfies \eqref{temp_w} (with constant $2C_0$), we have the bounds
$$
|f(r)-f(rt)|\le |f(r)| + |f(rt)| \leq C_1 r^k
$$
and
$$
\left|f(r)-f\left(\frac{r}{t}\right)\right|t^{-n+2s}\le \left( |f(r)| + \left|f\left(\frac{r}{t}\right)\right|\right) t^{-n+2s}\leq  C_1 \frac{r^k(1+t^{-k})}{t^{n-2s}},
$$
where we replace $r^k$ by $r^{-\delta}$ for $0<\delta\ll 1$ if $k=0$. 
Plugging these bounds into \eqref{frac bound} and using the fact that $k>-n$, if $k\neq 0$, we obtain 
$$
|(-\Delta)^s f(r)|\le C(n,s,C_1) r^{k-2s} \quad\text{ for any } r\in (0,1),
$$
and due to the assumption $s<\frac{n+k}{2}$, we have  $(-\Delta)^s f \in L^1(B_1)$. And if $k=0$, we have $|(-\Delta)^s f(r)|\le C(n,s,C_1) r^{-\delta-2s}$ for all $\delta>0$, which still gives $(-\Delta)^s f \in L^1(B_1)$ if we set $\delta\ll 1$ sufficiently small. 

Moreover, because $f$ is compactly supported and smooth, we have that for any $s \in (0,1)$,
$$
|(-\Delta)^s f(r)|\le C(n,s,|f|_{C^2},|\supp f|) r^{-n-2s}\qquad\mbox{for any $r>1$.}
$$
which implies $(-\Delta)^s \in L^1(B_1^c)$. Combining this with the previous bound in $B_1$ finishes the proof of the lemma.
\end{proof}

\end{appendix}

\FloatBarrier

\bibliographystyle{abbrv}

\bibliography{ref}

\end{document}